\documentclass[11pt,reqno]{amsproc}
\title[]{Long Time Dynamics of Nonequilibrium Electroconvection}

\author{Fizay-Noah Lee}
\address{Program in Applied and Computational Mathematics, Princeton University, Princeton, NJ 08544}
\email{fizaynoah@princeton.edu}
\usepackage[margin=1in]{geometry}
\usepackage{amsmath, amsthm, amssymb}
\usepackage{times}
\usepackage{color}
\usepackage{hyperref}
\usepackage{comment}
\usepackage{enumerate}
\usepackage{setspace}
\newcommand{\bA}{\mathbb{A}}
\newcommand{\ep}{\epsilon}

\newcommand{\pa}{\partial}

\newcommand{\la}{\label}
\newcommand{\fr}{\frac}
\newcommand{\na}{\nabla}
\newcommand{\be}{\begin{equation}}
\newcommand{\ee}{\end{equation}}
\newcommand{\bal}{\begin{aligned}}
\newcommand{\eal}{\end{aligned}}
\newcommand{\ba}{\begin{array}{l}}
\newcommand{\ea}{\end{array}}

\newtheorem{thm}{Theorem}
\newtheorem{prop}{Proposition}
\newtheorem{lemma}{Lemma}
\newtheorem{lem}{Lemma}
\newtheorem{rem}{Remark}
\newtheorem{defi}{Definition}
\newtheorem{cor}{Corollary}
\newcommand{\beg}{\begin}
\renewcommand{\div}{{\mbox{div}\,}}
\newcommand{\D}{\Delta}

\newcommand{\ug}{\underline\gamma}
\newcommand{\og}{\overline\gamma}
\newcommand{\um}{\underline M}
\newcommand{\om}{\overline M}

\newcommand{\BC}{\eqref{gamma}-\eqref{noslip}}

\newcommand{\mA}{\mathcal{A}}
\newcommand{\mV}{\mathcal{V}}
\newcommand{\mH}{\mathcal{H}}
\newcommand{\mB}{\mathcal{B}}
\newcommand{\mL}{\mathcal{L}}

\newcommand{\mF}{\mathcal{F}}
\newcommand{\tr}{\text{Tr}}

\newcommand{\rw}{{\mathring w}}
\newcommand{\rc}{{\mathring c}}
\newcommand{\ru}{{\mathring u}}
\newcommand{\rp}{{\mathring \Phi}}
\newcommand{\rr}{{\mathring\rho}}

\newcommand{\mVo}{\mathcal{V}_0}
\date{today}

\begin{document}

\noindent

\begin{abstract}
    The Nernst-Planck-Stokes (NPS) system models electroconvection of ions in a fluid. We consider the system, for two oppositely charged ionic species, on three dimensional bounded domains with Dirichlet boundary conditions for the ionic concentrations (modelling ion selectivity), Dirichlet boundary conditions for the electrical potential (modelling an applied potential), and no-slip boundary conditions for the fluid velocity. In this paper, we obtain quantitative bounds on solutions of the NPS system in the long time limit, which we use to prove 1) the existence of a compact global attractor with finite fractal (box-counting) dimension and 2) space-time averaged electroneutrality $\rho\approx 0$ in the singular limit of Debye length going to zero, $\ep\to 0$.
\end{abstract}

\maketitle

\section{Introduction}
We consider the Nernst-Planck-Stokes (NPS) system on an open connected bounded domain $\Omega\subset\mathbb{R}^3$ with smooth boundary. This system models electroconvection of ions in a fluid in the presence of boundaries. In this paper we focus on the case of two oppositely charged ionic species (valences $\pm1$). The full system is then given by the Nernst-Planck equations
\be
\bal
\pa_t c_1+u\cdot\na c_1&=D_1\div(\na c_1+c_1\na\Phi)\\
\pa_t c_2+u\cdot\na c_2&=D_2\div(\na c_2-c_2\na\Phi)\la{np}
\eal
\ee
coupled to the Poisson equation
\be
-\epsilon\D\Phi=c_1-c_2=\rho\la{pois}
\ee
and to the time dependent Stokes system
\be
\pa_t u-\nu\D u+\na p=-K\rho\na\Phi,\quad \div u=0.\la{stokes}
\ee
Above, $c_1$ and $c_2$ are the local ionic concentrations of the cation and anion, respectively, $\rho$ is a rescaled local charge density, $u$ is the fluid velocity, and $\Phi$ is a rescaled electrical potential. The positive constant $K$ is a coupling constant given by the product of Boltzmann's constant and the absolute temperature. The positive constants $D_i$ are the ionic diffusitives, and $\epsilon$ is a rescaled dielectric permittivity of the solvent and is proportional to the square of the Debye length. The Debye length in typical electrolytes (e.g. water) is very small, on the order of a few nanometers. Lastly, $\nu>0$ is the kinematic viscoscity of the fluid. The dimensional counterparts of $\Phi$ and $\rho$ are given by $(K/e)\Phi$ and $e\rho$, respectively, where $e$ is elementary charge. 

According to \eqref{np}, ionic concentrations are transported by the fluid, diffuse under their own concentration gradients, and are transported by the electrical field, which in turn, by \eqref{pois}, is generated by the charge density. The fluid itself is also forced by the electrical field, as indicated by the nonlinear forcing term in \eqref{stokes}.

Electroconvective systems like NPS have widespread applications in many fields of biology and chemistry and in engineering sciences (e.g. semiconductors, desalination processes, electrodialysis, etc.). We refer the reader to the books \cite{mock,prob,rubibook} for discussions on the applications of ionic electrodiffusion and electroconvection and for the physical derivation of the equations \eqref{np}-\eqref{stokes}.

For boundary conditions, we distinguish, as in \cite{NPS}, between equilibrium and nonequilibrium boundary conditions. \textit{Equilibrium boundary conditions} are boundary conditions on $c_i$ and $\Phi$ that admit steady state NPS solutions with \textit{constant} electrochemical potentials, defined by
\be
\mu_1=\log c_1+\Phi,\quad \mu_2=\log c_2-\Phi.
\ee
In this case, these steady state solutions are also unique. Equilibrium boundary conditions include the case where $c_i$ obey \textit{blocking (no-flux) boundary conditions}
\be
{\pa_n\mu_i}_{|\pa\Omega}=0,\quad i=1,2
\ee
and $\Phi$ obey inhomogeneous Dirichlet, Neumann, or Robin boundary conditions. Also included is the case where $\Phi$ obeys Dirichlet boundary conditions and $c_i$ obey blocking boundary conditions on boundary portions $S_i\subset\pa\Omega$ while satisfying Dirichlet boundary conditions on the complements $\pa\Omega\setminus S_i$ in such a way that
\be
{\mu_1}_{|\pa\Omega\setminus S_1}=\text{constant},\quad {\mu_2}_{|\pa\Omega\setminus S_2}=\text{constant}.\la{us}
\ee
We remark that $S_1$ and $S_2$ need not be identical, and we allow for $S_i\in\{\emptyset,\pa\Omega\}.$ In all the cases mentioned above, we take no-slip boundary conditions for the fluid velocity $u_{|\pa\Omega}=0$.

In the mathematics literature, many authors have considered various Nernst-Planck models in the context of equilibrium boundary conditions. For the uncoupled Nernst-Planck equations, global existence of strong solutions is proved in \cite{biler, choi} for various equilibrium boundary conditions with blocking boundary conditions for $c_i$. Global existence of strong solutions of the two dimensional Nernst-Planck-Navier-Stokes (NPNS) system with various equilibrium boundary conditions is proved in \cite{bothe, ci,schmuck}. A common feature in analyzing Nernst-Planck/NPNS with equilibrium boundary conditions is the existence of a natural dissipative energy inequality, which gives a priori bounds, from which higher regularity bounds are obtained through boostrapping. This same inequality is also used in \cite{np3d,ryham} and \cite{fnl} to obtain small data global regularity for three dimensional NPNS and large data global regularity for three dimensional NPS, respectively. 

For equilibrium boundary conditions, the corresponding unique steady state solution (\textit{Boltzmann state} solution) is characterized by zero fluid flow $u^*\equiv 0$ and concentrations $c_i^*$, related to $\Phi^*$ by
\be
c_1^* = Z_1^{-1} e^{-\Phi^*},\quad c_2^* = Z_2^{-1} e^{\Phi^*}\la{c*}
\ee
with $Z_i>0$ appropriate constants, determined by the (equilibrium) boundary conditions and/or initial conditions \cite{ci,np3d}. Here, $\Phi^*$ is the unique solution to a nonlinear Poisson equation, called the Poisson-Boltzmann equation
\be
-\epsilon\D\Phi^* = Z_1^{-1} e^{-\Phi^*}-Z_2^{-1} e^{\Phi^*}\la{PB}
\ee
with the relevant boundary conditions. In two dimensions, these Boltzmann states are globally asymptotically stable - that is, starting from arbitrarily large initial conditions, solutions converge towards Boltzmann states. In three dimensions, Boltzmann states are at least locally asymptotically stable - solutions converge towards the relevant Botlzmann state, provided they are initially sufficiently close to them. These facts are established for NPNS in \cite{ci} (2D) and \cite{np3d} (3D).


In full generality, one may define {nonequilibrium boundary conditions} in the negation - that is, \textit{nonequilibrium boundary conditions} are boundary conditions that do \textit{not} admit steady state NPS solutions with constant electrochemical potentials. A concrete example is given by $c_1$ and $\Phi$ satisfying Dirichlet boundary conditions on a nonempty boundary portion $S\subset\pa\Omega$ in such a way that $(\log  c_1+\Phi)_{|S}$ is not constant.

Relative to the equilibrium case, there are much fewer rigorous mathematical results for nonequilibrium boundary conditions. On the other hand, physically, nonequilibrium boundary conditions are linked to various interesting phenomena. For example, both numerical \cite{davidson,pham} and experimental \cite{kang,rubinstein} evidence exists for the development of instabilities to NPNS/NPS systems associated with nonequilibrium boundary conditions, whereby for boundary conditions sufficiently far from equilibrium, vortical and even chaotic flow patterns are observed adjacent to the fluid/boundary interface. Such observations are reminiscent of the pattern formation and thermal turbulence associated with Rayleigh-Bénard convection. Despite the prevalence and widespread application of this so-called \textit{electrokinetic instability}, our theoretical understanding of this phenomenon is far from complete. On one hand, the question of the exact mechanism behind the development of these instabilities is yet to be fully resolved (c.f. classical works \cite{rubizaltz,zaltzrubi} and also the works \cite{pham,rubisegel}). On the other, a mathematically rigorous qualitative picture of the long time dynamics of solutions of NPNS/NPS for nonequililbrium boundary conditions \textit{after} the onset of instability is also, to this date, unavailable.  

In this paper, our focus is on the long time regime, long after the onset of electrokinetic instability. For nonequilibrium boundary conditions, however, even the question of global existence of strong solutions of NPNS/NPS is not resolved in full generality. The works \cite{ci} (2D NPNS) and \cite{cil,fnl} (3D NPS) establish global regularity for large classes of nonequilibrium boundary conditions. In particular, in \cite{cil}, the authors consider nonequilibrium boundary conditions whereby both the concentrations $c_i$ and $\Phi$ satisfy \textit{arbitrary} Dirichlet boundary conditions, together with no-slip for $u$:
\begin{align}
{c_i}_{|\pa\Omega}&=\gamma_i>0,\quad i=1,2\la{gamma}\\
\Phi_{|\pa\Omega}&=W\la{W}\\
u_{|\pa\Omega}&=0.\la{noslip}
\end{align}
We observe that other than for special choices of $\gamma_i$ and $W$, the relation $\log \gamma_1+\Phi=$ constant, $\log \gamma_2-\Phi=$ constant does not hold, implying that, generally, boundary conditions \eqref{gamma}-\eqref{noslip} are nonequilibrium.

In this paper, we consider the same class of boundary conditions as in \cite{cil}, given by \eqref{gamma}-\eqref{noslip}. Also, for simplicity, we assume $\gamma_i$ and $W$ are smooth. Beyond the aforementioned global regularity however, there are few rigorous results on the long time dynamics of solutions of NPS or NPNS with nonequilibrium boundary conditions \eqref{gamma}-\eqref{noslip}; the goal of this paper is to study this long time behavior. 

In \cite{abdo}, the authors establish the existence of a finite dimensional (i.e. finite fractal (box-counting) dimension) global attractor for NPNS on the two dimensional torus, in the presence of a time independent external force acting on the fluid velocity $u$ (in addition to the time dependent nonlinear electrical forcing term). A global attractor is a compact subset of phase space that is invariant under the solution map and attracting in the sense that all time dependent solutions are drawn arbitrarily close to it (in the appropriate phase space topology) in the long time limit. However, insofar as it is our goal to analyze the long time behavior of NPS solutions beyond the onset of electrokinetic instability, it is important to account for the presence of boundaries as this is a phenomenon that takes place, empirically, in a boundary layer. This motivates our first main result, in Section \ref{FD}, which establishes the existence of a finite dimensional global attractor for NPS on an arbitrary open, connected, bounded domain of $\mathbb{R}^3$ with smooth boundary, with boundary conditions \eqref{gamma}-\eqref{noslip}. Our results (Theorems \ref{globa} and \ref{fdf}) are the first such results for the NPS system on a domain with boundary. The existence of a finite dimensional global attractor, in particular, implies that the long time trajectories of NPS solutions are limited to a finite dimensional manifold, \textit{independent} of initial conditions, thus simplifying, at least on a theoretical level, the long time behavior of electroconvective flow, even in the emergence of the turbulent behavior observed experimentally/numerically.

It is worth noting that, at this stage, it is an open problem as to whether the NPS system has a global attractor for \textit{arbitrary} (but sufficiently regular) nonequiliubrium boundary conditions. For example, while there are other classes of nonequilibrium boundary conditions for which unique, global strong solutions are known to exist \cite{ci,fnl}, the analysis of this paper does not naturally extend to these other cases (these are cases that may be described as having \textit{mixed} boundary conditions; say, the ionic concentrations satisfying a mixture of Dirichlet and blocking boundary conditions). For our boundary conditions \eqref{gamma}-\eqref{noslip}, the starting point of our proof of the existence of a global attractor is the energy inequality given in Lemma \ref{l1}, which gives absorbing ball, and in particular \textit{uniform in time}, bounds on solutions of NPS (this lemma is taken from \cite{cil}, where it is used to prove global regularity of the NPS system). This situation is in contrast to the cases in \cite{ci,fnl} where the corresponding bounds are at least \textit{exponential in time}. Determining whether or not a global attractor exists in these other cases, then, critically depends on whether or not these exponential/superexponential in time bounds are sharp or can be improved. This remains an open task.

We also remark that an energy inequality similar to that of Lemma \ref{l1} is also used in \cite{abdo} to prove the existence of a global attractor for NPNS on the two dimensional torus. One critical difference, however, is that the energy functional of Lemma \ref{l1} is a linear combination, with precisely chosen coefficients (c.f. \eqref{mathcalf}), of various norms, including the $L^2$ norms of the ionic concentrations $c_i$, the kinetic energy $\|u\|_{L^2}^2$ and the $H^{-1}$ norm of the charge density $\rho$. It it the presence of boundaries that necessitates the use of this specific energy functional, a need that does not arise in the absence of boundaries.

Section \ref{EN} of this paper is devoted to a feature of charged fluids called electroneutrality, which refers to the fact that at distances larger than the Debye length from charged boundaries, roughly neutral charge density is maintained $|\rho| \ll1$. Formally, electroneutrality is also a natural expectation given that $\ep$ is a small parameter (c.f. \eqref{pois}). This leads to the study of the NPS system in the limit of $\ep\to 0$. The mathematical challenge posed by this limit is the fact that $\ep$ is a coefficient of a second order operator, thus making the limit singular.

In the case of equilibrium boundary conditions (heuristically, when no electric current is flowing through the fluid), electroneutrality is verified rigorously in \cite{EN, hsieh}. The mathematical statement of electroneutrality for equilibrium boundary conditions is that
\be
\lim_{\epsilon\to 0}\lim_{t\to\infty}\sup_{x\in K}|\rho_\epsilon(t,x)|=0.\la{inten}
\ee
Above the $\epsilon$ subscript indicates the dependence of $\rho$ on the value of $\epsilon$, and $K\subset\Omega$ is any compact subset. Thus, to be more precise, \eqref{inten} establishes uniform pointwise electroneutrality, away from boundaries, in the long time limit. The proof of \eqref{inten} in \cite{EN} follows a two-step analysis. First, in equilibrium, the existence of a dissipative energy inequality implies that the global attractor of $c_{i,\epsilon}(t,x)$ is a singleton (c.f. \eqref{c*}), and thus $\rho_\epsilon(t,x)\to_{t\to\infty}\rho_\epsilon^*(x)$, where $\rho_\epsilon^*$ is determined uniquely by the prescribed data. The second step is to then analyze the steady state charge density $\rho_\epsilon^*$ and to establish $\lim_{\epsilon\to 0}\sup_{x\in K}|\rho_\epsilon^*(x)|=0$ on compact subsets $K\subset\Omega$.

Now the question we consider is whether electroneutrality is just an equilibrium feature. For general Dirichlet boundary conditions on $c_i$ that are nonequilibrium, it is not known and not expected that the global attractor is a singleton. So, it is not possible to identically replicate the two-step analysis of \cite{EN} to prove an electroneutrality statement. Nonetheless, the existence of a compact global attractor allows us to replace the analysis on a fixed steady state with the analysis of general elements of the global attractor obtained in Section \ref{FD}. Due to the added level of generality, we do not expect a result as strong as \eqref{inten} to hold. However, we prove the following weaker statement, which establishes electroneutrality in a space-time averaged sense,
\be
\lim_{\epsilon\to 0}\lim_{T\to\infty}\fr{1}{T}\int_0^T\left(\int_\Omega \rho_\epsilon^2(t,x)\,dx\right)\,dt=0.
\ee
We emphasize that the above result (see Theorem \ref{aen} for the precise statement) holds for arbitrary Dirichlet data on $c_i$, which in particular includes cases where, based on experiments and numerical simulations, the time dependent solutions are expected to exhibit unstable and even chaotic behavior. While in spirit, the proof of this space-time averaged electroneutrality is a consequence of the absorbing ball properties established in Section \ref{FD}, the actual proof uses a different absorbing ball estimate, with bounds in $L^\infty$. Namely, we use a result proved in \cite{NPS}, which gives long time upper and lower pointwise bounds on the ionic concentrations $c_i$ in terms of only the boundary data (Theorem \ref{maxthm}). In partciular, these bounds are independent of $\epsilon$. In contrast, the bounds obtained in Section \ref{FD} are in $L^2$ based Sobolev spaces and depend on $\epsilon$ in such a way that they become unbounded in the limit of $\epsilon\to 0$. To the best of our knowledge, no prior mathematical work exists on electroneutrality under nonequilibrium conditions.

We note that our focus on the NPS system, as opposed to the NPNS system, is physically relevant in the sense that the typical length scale of electroconvective systems is small enough to justify the use of the Stokes equations. However, this choice is also due to our choice of spatial dimension (i.e. three) and the fact that we have global regularity in the case of three dimensional NPS for our boundary conditions (see Theorem \ref{twospecies} below). This is not the case for three dimensional NPNS, due to coupling with the Navier-Stokes equations for which global regularity in three dimensions, as in well known, is an open problem. The results of this paper, then, hold for NPNS in two dimensions, where global regularity is known and for which adequate quantitative bounds on the velocity $u$ are available. The proof of the corresponding results in this case follow from the computations of this paper, with straightforward modifications having to do with the added presence of the nonlinear term in the Navier-Stokes equations; however, we shall not pursue this extension here.

This paper is structured is follows. In Section \ref{pre}, we introduce the relevant functional setting along with various notation. In Section \ref{FD}, we prove the existence of a finite dimensional global attractor to the NPS system. In Section \ref{EN} we show that electroneutrality holds for NPS in a space-time averaged sense.

\section{Preliminaries}\la{pre}
Unless otherwise stated, we denote by $C$ a positive constant that depends only on the parameters of the system, the domain, and the boundary conditions. The value of $C$ may differ from line to line. When constants depend on the initial data, we explicitly indicate this fact. Also, when there is no cause for confusion, we write $A\lesssim B$ to mean $A\le CB$ for some constant $C$ with the aforementioned dependencies.

We denote by $L^p=L^p(\Omega)$ the Lebesgue spaces and by $W^{s,p}=W^{s,p}(\Omega)$, $H^s=H^s(\Omega)=W^{s,2}$ the Sobolev spaces. We denote by $(\cdot,\cdot)_{L^2}$ the $L^2(\Omega)$ inner product. We denote by $H_0^1=H_0^1(\Omega)$ the function space of $H^1(\Omega)$ functions with zero trace. We endow $H_0^1$ with the Dirichlet inner product, $(f,g)_{H_0^1}=\int_\Omega \na f\cdot\na g\,dx.$ We remark that for $f,g\in H^2\cap H_0^1$, we have $(f,g)_{H_0^1}=(f,-\D g)_{L^2}=(g,-\D f)_{L^2}.$

We define $H$ and $V$ to be the closures of $\{f\in (C_0^\infty(\Omega))^3\,|\,\div f=0\}$ in $(L^2(\Omega))^3$ and $(H^1(\Omega))^3$, respectively. We endow $H$ with the $L^2$ inner product and $V$ with the Dirichlet inner product, $( f,g)_V=\int_\Omega \na f:\na g\,dx.$ 

To avoid dealing directly with the pressure $p$, it is oftentimes convenient to do computations on the Stokes equations \eqref{stokes} projected onto the space of divergence free functions. We denote by $\mathbb{P}:(L^2(\Omega))^3\to H$ the Leray projection which is the orthogonal projection of $L^2(\Omega)^3$ onto $H$ (see \cite{cf} for more details). Then, applying $\mathbb{P}$ to \eqref{stokes}, we obtain
\be\la{Stokes}
\pa_t u+\nu Au=-K\mathbb{P}(\rho\na\Phi)
\ee
where $A=\mathbb{P}(-\D):\mathcal{D}(A)=(H^2)^3\cap V\to H$ is the Stokes operator whose inverse $A^{-1}$ is a self-adjoint, nonnegative compact operator on $H$. From elliptic regularity estimates for $A$ (see e.g. \cite{cf}), we also have that $A^{-1}$ is a self-adjoint, nonnegative compact operator on $V$. We remark that for $f,g\in \mathcal{D}(A)$, we have that $(f,g)_V=(f,Ag)_H=(g,Af)_H$. This follows from the self-adjointness of $\mathbb{P}$ and the fact that for $f,g\in \mathcal{D}(A)$, we have $\mathbb{P}f=f$ and $\mathbb{P}g=g$.

We define $\mathcal{H}=L^2\times L^2\times H$, $\mathcal{V}= H^1\times H^1\times V$, and $\mathcal{H}^2=H^2\times H^2\times \mathcal{D}(A)$ and endow them with the norms
\begin{align*}
\|(f,g,h)\|_\mathcal{H}^2=&\|f\|_{L^2}^2+\|g\|_{L^2}^2+\|h\|_{H}^2\\
\|(f,g,h)\|_\mathcal{V}^2=&\|f\|_{H^1}^2+\|g\|_{H^1}^2+\|h\|_{V}^2\\
\|(f,g,h)\|_{\mathcal{H}^2}^2=&\|f\|_{H^2}^2+\|h\|_{H^2}^2+\|Ah\|_{H}^2.
\end{align*}
We also denote $\mathcal{V}_0=H_0^1\times H_0^1\times V$ and endow the space with the norm
\be
\|(f,g,h)\|_{\mVo}^2=\|f\|_{H_0^1}^2+\|g\|_{H_0^1}^2+\|h\|_V^2,
\ee
and denote $\mathcal V^+=\{(f,g,h)\in \mV\,|\, f,g\ge 0\}\subset \mV$.

For functions of both the spatial variable $x$ and time $t$, we say that $f\in L^p(0,T;X)$ if the Banach space valued function $t\mapsto f(t,\cdot)\in X$ is in $L^p(0,T)$.

We denote $z_1=1, z_2=-1$ for the valences of the ionic species.

\section{Existence and Finite Dimensionality of Global Attractor}\la{FD}
The global existence and regularity of solutions of the NPS system \eqref{np}-\eqref{stokes} with Dirichlet boundary conditions \eqref{gamma}-\eqref{noslip} is established in \cite{cil}. In particular, the following theorem is proved.

\beg{thm}\la{twospecies} Let $T>0$ be arbitrary. Let initial conditions $c_i(\cdot, 0)\ge 0$, $c_i(\cdot, 0)\in H^1$ and $u_0\in V$, and smooth boundary conditions ${c_i}_{|\pa\Omega}=\gamma_i>0$ and $\Phi_{|\pa\Omega}=W$ be given. Then the NPS system \eqref{np}-\eqref{stokes} has a unique global strong solution on $[0,T]$. That is, there exist unique $0\le c_i\in L^\infty(0,T;H^1)\cap L^2(0,T;H^2), u\in L^\infty(0,T;V)\cap L^2(0,T ;\mathcal{D}(A))$ that solve \eqref{np}-\eqref{stokes} in the sense of distributions and satisfy the boundary conditions \eqref{gamma}-\eqref{noslip} in the sense of traces.
\end{thm}

Thus, for fixed, smooth boundary conditions ${c_i}_{|\pa\Omega}=\gamma_i>0$ and $\Phi_{|\pa\Omega}=W$, we define the \textit{solution map} $S(\cdot,\cdot)$, seen as a map from $[0,\infty)\times \mathcal{V}^+$ to $\mathcal{V}^+$, such that $(t,x)\mapsto S(t,w_0)(x)$ is the unique solution to the initial-boundary value problem \eqref{np}-\eqref{stokes} with initial conditions $w_0\in\mathcal{V}^+$ and satisfying the fixed boundary conditions. By some abuse of notation, we also denote $S(t)$ to be the map from $\mathcal{V}^+$ to itself that maps $w_0\in\mathcal{V}^+$ to $S(t,w_0)\in\mathcal{V}^+$. We note that due to the uniqueness of strong solutions to the NPS system, the map $S$ satisfies the semigroup property
\be
S(t+s)=S(t)S(s)=S(s)S(t),\quad t,s\ge 0.
\ee

The goal of this section is to show the existence and finite dimensionality of a global attractor associated with the NPS system.

\begin{defi}\la{GA}
We say that $\mathcal{A}\subset \mathcal{V}^+$ is a global attractor of the NPS system \eqref{np}-\eqref{stokes} with boundary conditions \eqref{gamma}-\eqref{noslip} if $\mA$ satisfies the following:
\begin{enumerate}
    \item $\mathcal{A}$ is compact in $\mathcal V$
    \item $\mathcal{A}$ is invariant under $S$ i.e. $S(t)\mathcal{A}=\mathcal{A}$ for all $t\ge 0$
    \item\la{maximal} $\mathcal{A}$ is maximal in the sense that if $\mathcal B\subset\mV^+$ is bounded in $\mathcal{V}$ and satisfies $S(t)\mathcal B=\mathcal B$ for all $t\ge 0$, then $\mathcal B\subset \mathcal A$
    \item \la{att}For all $v\in \mathcal{V}^+$, $\lim_{t\to\infty}d_\mathcal{V}(S(t)v,\mathcal{A})=0$, where $d_\mathcal{V}(S(t)v,\mathcal{A}):=\inf_{x\in\mathcal{A}}\|S(t)v-x\|_{\mathcal V}$
    \item $\mathcal{A}$ is connected in $\mathcal{V}.$
\end{enumerate}
\end{defi}

We establish the existence and finite dimensionality of a global attractor of the NPS system in two separate theorems.

\begin{thm}\la{globa}
The NPS system \eqref{np}-\eqref{stokes}  with boundary conditions \eqref{gamma}-\eqref{noslip} possesses a global attractor $\mathcal A$.
\end{thm}

\begin{thm}\la{fdf}
The global attractor $\mA$ has finite fractal (box-counting) dimension (c.f. \eqref{fd}) in $\mV$:
    \be 
    d_f(\mA)\le D<\infty
    \ee
    where $D$ depends only on parameters and boundary conditions.
\end{thm}

From here on, for initial conditions $w_0=(c_1(0),c_2(0),u(0))\in \mV^+$, we denote
\be
w(t,x)=(c_1(t,x),c_2(t,x),u(t,x))=S(t,w_0)(x).
\ee

\subsection{Proof of Theorem \ref{globa}.}

To prove Theorem \ref{globa}, we first establish some properties of the solution map $S$.

\begin{prop}\la{p1}The solution map $S$ satisfies
\begin{enumerate}[(I)]
    \item\la{ab} (\textit{absorbing ball}) for initial conditions $w_0\in \mV^+$ there exist $T_0=T_0(\|w_0\|_\mathcal{H})>0$ and $R>0$ independent of $w_0$ such that for all $t\ge T_0$, $S(t,w_0)\in B^2_R=\{v\in\mathcal{V}^+\,|\,\|v\|_{\mathcal{H}^2}\le R\}$
    \item\la{ab1} (\textit{continuity}) for all $t\ge 0$, $S(t):\mathcal{V}^+\to\mathcal{V}^+$ is continuous in the $\mV$ topology
    \item\la{ab2} (\textit{injectivity}) for all $t\ge 0$, $S(t):\mathcal{V}^+\to\mathcal{V}^+$ is injective.
\end{enumerate}
\end{prop}

To simplify the proof of the proposition, it is helpful to invoke the following result, the first half of which is found in \cite{cil}:

\begin{lemma}\la{l1}
For the NPS system \eqref{np}-\eqref{stokes} with boundary conditions \eqref{gamma}-\eqref{noslip}, there exist constants $\delta, C_j>0$, $j=1,2,3,4$ depending only on parameters and boundary conditions such that the energy
\be
\mathcal{F}=\fr{1}{2K}\|u\|_H^2+\mathcal{P}+\delta\sum_{i=1}^2\|c_i\|_{L^2}^2\la{mathcalf}
\ee
satisfies
\be
\fr{d}{dt}\mathcal{F}+C_1\|u\|_V^2+C_2\sum_{i=1}^2\|\na c_i\|_{L^2}^2+C_3\|\rho\|_{L^3}^3\le C_4\la{Fineq}
\ee
where
\be
\mathcal{P}=\fr{1}{2\epsilon}\int_\Omega \rho(-\D_D)^{-1}\rho\,dx\ge 0
\ee
and $-\D_D$ is the homogeneous Dirichlet Laplace operator on $\Omega$.

In addition, for constants $C_i$, $i=5,6,7$ depending only on parameters and boundary conditions, we have for all $t,\tau\ge 0$,
\be\la{FB}
\mathcal{F}(t+\tau)\le \mathcal{F}(t)e^{-C_5\tau}+C_6\le \mathcal{F}(t)+C_6
\ee
and
\be\la{fint}
\int_{t}^{t+\tau}\left(\|u(s)\|_V^2+\sum_{i=1}^2\|\na c_i(s)\|_{L^2}^2+\|\rho(s)\|_{L^3}^3\right)\,ds\le C_7(\mathcal{F}(t)+\tau).
\ee
In particular, there exist $R_0>0$ depending only on parameters and boundary conditions and $T_1=T_1(\|w_0\|_\mathcal{H})>0$ such that
\be
\mathcal{F}(t)\le R_0,\quad \forall t\ge T_1.\la{qball}
\ee\la{balllem}
\end{lemma}

\begin{proof}
The proof of \eqref{Fineq} may be found in \cite{cil}. Next, \eqref{FB} follows from \eqref{Fineq} by using the Poincaré-type bound $\mathcal{F}\lesssim\sum_i\|\na c_i\|_{L^2}^2+\|u\|_V^2+1$ and applying Grönwall's inequality; \eqref{fint} follows from integrating \eqref{Fineq} in time. Lastly, \eqref{qball} follows from the first inequality of \eqref{FB} by taking $t=0$ and taking $\tau$ sufficiently large.
\end{proof}

\begin{rem}
We remark that for \eqref{qball}, the time $T_1$ is obtained from \eqref{FB} and requiring that $\mathcal{F}(0)e^{-C_5 T_1}+C_6\le R_0$. Thus, without loss of generality, $T_1$ may be chosen to be a monotone increasing function of $\mathcal{F}(0)$. Also, the proof of Proposition \ref{p1} \eqref{ab} reveals that we may take $T_0=T_1+2$, and so $T_0$ may also be chosen to be a monotone increasing function of $\mF(0)$. \la{monotone}
\end{rem}

We now prove Proposition \ref{p1}.

\begin{proof}
We begin with \eqref{ab}.

\textbf{Step 1. Absorbing ball for $\|u\|_V$.} With $T_1>0$ determined by Lemma \ref{balllem}, we fix $t\ge T_2=T_1+1$. Then, from \eqref{fint} we find
\be
\int_{t-1}^{t}\left(\|u(s)\|_V^2+\sum_{i=1}^2\|\na c_i(s)\|_{L^2}^2+\|\rho(s)\|_{L^3}^3\right)\,ds\lesssim \mathcal{F}(t-1)+1.
\ee
Since $t-1\ge T_1$, we have that by Lemma \ref{balllem}, $\mathcal{F}(t-1)\le R_0$ where $R_0$, from \eqref{qball}, depends only on parameters and boundary data. Since $t\ge T_2$ was arbitrary, we have shown
\be
\int_{t-1}^t\left(\|u(s)\|_V^2+\sum_{i=1}^2\|\na c_i(s)\|_{L^2}^2+\|\rho(s)\|_{L^3}^3\right)\,ds\le R_1,\quad t\ge T_2\la{Ru'}
\ee
for $R_1$ depending only on parameters and boundary data. 

Then, we multiply \eqref{Stokes} by $Au$, integrate by parts, and apply Young's inequality to obtain
\be
\fr{d}{dt}\|u\|_V^2+\nu\|Au\|_H^2\lesssim\|\rho\|_{L^3}^2\|\na\Phi\|_{L^6}^2\lesssim (\|\rho\|_{L^3}^3+1)(\|\rho\|_{L^2}^2+1)\la{ddtu}
\ee
where in the second inequality, we used \eqref{pois} and the embedding $H^1\hookrightarrow L^6.$ By \eqref{qball}, we have
\begin{align}
\|\rho(t)\|_{L^2}^2\le& R_2,\quad t\ge T_1=T_2-1\la{Rr}
\end{align}
for $R_2$ depending only on parameters and boundary data. Now, we fix $t\ge T_2$ and then fix $t_0\in (t-1,t)$ such that $\|u(t_0)\|_V^2\le R_1$. The existence of such a $t_0$ is guaranteed by \eqref{Ru'}. Then, integrating \eqref{ddtu} from $t_0$ to $t$ and using \eqref{Ru'} again, we find
\be
\|u(t)\|_V^2\le \|u(t_0)\|_V^2+C(R_2+1)\int_{t_0}^t\|\rho(s)\|_{L^3}^3+1\,ds\le R_1+C(R_2+1)(R_1+1).
\ee
Thus we have shown
\be
\|u(t)\|_V^2\le R_3=  R_1+C(R_2+1)(R_1+1),\quad t\ge T_2\la{Ru}
\ee
where $R_3$ is a constant depending only on parameters and boundary data.

\textbf{Step 2. Absorbing ball for $\|c_i\|_{H^1}.$}
Multiplying \eqref{np} by $-\D c_i$ and integrating by parts and using Hölder inequalities, we obtain
\be
\bal
&\fr{1}{2}\fr{d}{dt}\|\na c_i\|_{L^2}^2+D_i\|\D c_i\|_{L^2}^2\\
\lesssim&(\|u\|_{L^6}\|\na c_i\|_{L^3}+\|\na c_i\|_{L^3}\|\na \Phi\|_{L^6}+\|c_i\|_{L^6}\|\rho\|_{L^3})\|\D c_i\|_{L^2}.\la{ddtq}
\eal
\ee
Combining elliptic, Sobolev, and interpolation inequalities (interpolating $L^3$ between $L^2$ and $H^1$) with Young's inequalities, we obtain from \eqref{ddtq}
\be
\bal
\fr{d}{dt}\|\na c_i\|_{L^2}^2+D_i\|\D c_i\|_{L^2}^2\lesssim\left(1+\|u\|_V^4+\|\rho\|_{L^2}^4+\|\rho\|_{L^3}^3\right)\|\na c_i\|_{L^2}^2.\la{ddtq2}
\eal
\ee
Fixing $t\ge T_2$, we use \eqref{Ru'} to find $t_0\in (t-1,t)$ such that 
\be
\|\na c_i(t_0)\|_{L^2}^2\le R_4\la{Rq'}
\ee
for some $R_4$ depending only on parameters and boundary data. Then applying a Grönwall inequality to \eqref{ddtq2} and using \eqref{Ru'}, \eqref{Ru} and \eqref{Rq'} together with Lemma \ref{balllem}, we find
\be
\bal
\|\na c_i(t)\|_{L^2}^2\le& \|\na c_i(t_0)\|_{L^2}^2\exp\left(C\int_{t_0}^t(1+\|u\|_V^4+\|\rho\|_{L^2}^4+\|\rho\|_{L^3}^3)\,ds\right)\\
\le& R_4\exp(C(1+R_3^2+R_0^2+R_1)).
\eal
\ee
Thus we have shown
\be
\|\na c_i(t)\|_{L^2}^2\le R_5=R_4\exp(C(1+R_3^2+R_0^2+R_1)),\quad t\ge T_2, \quad i=1,2\la{Rq}
\ee
where $R_5$ ultimately depends only on parameters and boundary data.

\textbf{Step 3. Uniform local time integrability of $\|\pa_t u\|_H^2$.}  Taking the inner product of \eqref{Stokes} with $\pa_t u$, integrating by parts and applying Young's inequality, we obtain
\be
\fr{d}{dt}\fr{\nu}{2}\|u\|_V^2+\fr{1}{2}\|\pa_tu\|_H^2\lesssim \|\rho\na\Phi\|_{L^2}^2\le\|\rho\|_{L^3}^2\|\na\Phi\|_{L^6}^2.\la{ddtnu}
\ee
By \eqref{Rq} and Sobolev embeddings, we see that the rightmost term in \eqref{ddtnu} is uniformly bounded on $[T_2,\infty)$ by a constant depending only on parameters and boundary data. Using this fact, together with \eqref{Ru}, we integrate \eqref{ddtnu} to find, 
\be
\bal
\int_{t-1}^t\|\pa_s u(s)\|_H^2\,ds\le& \nu\|u(t-1)\|_V^2+C\int_{t-1}^t\|\rho(s)\|_{L^3}^2\|\na\Phi(s)\|_{L^6}^2\,ds\\
\le&\nu R_3+C(R_5^2+1)\\
=&R_6\quad  \text{ for all } t\ge T_3=T_2+1\la{bRu}
\eal
\ee
where $R_6$ depends only on parameters and boundary data.

\textbf{Step 4. Uniform local time integrability of $\|\pa_t c_i\|_{L^2}^2$.} Multiplying \eqref{np} by $\pa_t c_i$, integrating by parts and applying Young's inequalities, we obtain
\be
\bal
\fr{d}{dt}\fr{D_i}{2}\|\na c_i\|_{L^2}^2+\fr{1}{2}\|\pa_t c_i\|_{L^2}^2\lesssim(\|u\cdot\na c_i\|_{L^2}^2+\|\na c_i\cdot\na\Phi\|_{L^2}^2+\|c_i\rho\|_{L^2}^2).\la{ddtc}
\eal
\ee
Then, we note that from \eqref{Stokes} and from elliptic and Stokes regularity estimates \cite{cf}, we have
\be
\|u\|_{L^\infty}\lesssim\|u\|_{H^2}\lesssim\|Au\|_{L^2}\lesssim(\|\pa_t u\|_H+\|\rho\na\Phi\|_{L^2}).\la{ui}
\ee
Thus from \eqref{ddtc}, we obtain
\be
\bal
\fr{d}{dt}\fr{D_i}{2}\|\na c_i\|_{L^2}^2+\fr{1}{2}\|\pa_t c_i\|_{L^2}^2\lesssim& (\|u\|_{L^\infty}^2\|\na c_i\|_{L^2}^2+\|\na c_i\|_{L^2}^2\|\na\Phi\|_{L^\infty}^2+\|c_i\|_{L^4}^2\|\rho\|_{L^4}^2)\\
\lesssim&R_5\|\pa_t u\|_H^2+R_5^2+R_5^3+1\\
=&R_7\|\pa_t u\|_H^2+R_8,\quad t\ge T_2\la{Rc}
\eal
\ee
where in the second line, we used \eqref{ui} and the fact that all the terms that occur, other than $\|\pa_t u\|_{L^2}^2$, are uniformly bounded on $[T_2,\infty)$ in terms of $R_5$ from \eqref{Rq}. The constants $R_7$ and $R_8$ also depend only on parameters and boundary conditions. Next, we integrate \eqref{Rc} to find
\be
\bal
\int_{t-1}^t\|\pa_s c_i(s)\|_{L^2}^2\,ds\le& D_i\|\na c_i(t-1)\|_{L^2}^2+2\int_{t-1}^t(R_7\|\pa_s u(s)\|_{H}^2+R_8)\,ds\\
\lesssim& R_5+R_6R_7+R_8\\
=&R_9,\quad t\ge T_3,\quad i=1,2\la{bRc}
\eal
\ee
where in the second line we used \eqref{bRu}.

\textbf{Step 5. Absorbing ball for $\|\pa_t c_i\|_{L^2}$ and $\|\pa_t u\|_{L^2}$.} Differentiating \eqref{Stokes} in time, taking the inner product of the resulting equation with $\pa_t u$,  integrating by parts, and using the elliptic estimate 
\be\la{ntp}
\|\na\pa_t\Phi\|_{L^6}\lesssim\|\pa_t c_1\|_{L^2}+\|\pa_t c_2\|_{L^2}
\ee
we obtain
\be
\fr{1}{2}\fr{d}{dt}\|\pa_t u\|_H^2+\fr{\nu}{2}\|\pa_t u\|_V^2\lesssim(\|\pa_t c_1\|_{L^2}^2+\|\pa_t c_2\|_{L^2}^2)(\|\na\Phi\|_{L^\infty}^2+\|\rho\|_{L^3}^2).\la{ptu1}
\ee
Restricting the above estimate to $t\ge T_2$ and using \eqref{Rq}, we obtain
\be
\fr{d}{dt}\|\pa_t u\|_H^2\lesssim (R_5+1)(\|\pa_t c_1\|_{L^2}^2+\|\pa_t c_2\|_{L^2}^2),\quad t\ge T_2.\la{ptu}
\ee
Similarly, differentiating \eqref{np} in time, multiplying the resulting equation by $\pa_t c_i$, integrating by parts, and using
\be
\int_\Omega (u\cdot\na\pa_t c_i)\pa_tc_i\,dx=0
\ee
we obtain, again using \eqref{ntp},
\be
\fr{1}{2}\fr{d}{dt}\|\pa_t c_i\|_{L^2}^2+\fr{D_i}{2}\|\na\pa_tc_i\|_{L^2}^2\lesssim \|\pa_t u\|_{H}^2\|c_i\|_{L^\infty}^2+(\|\pa_t c_1\|_{L^2}^2+\|\pa_t c_2\|_{L^2}^2)(\|\na\Phi\|_{L^\infty}^2+\|c_i\|_{L^3}^2).\la{ptc1}
\ee
From (\ref{np}), we deduce
\be\la{Dc}
\bal 
\|\D c_i\|_{L^2}\lesssim& \|\pa_t c_i\|_{L^2}+\|u\|_{L^\infty}\|\na c_i\|_{L^2}+\|\na c_i\|_{L^2}\|\na \Phi\|_{L^\infty}+\|c_i\|_{L^4}\|\rho\|_{L^4}\\
\lesssim&\|\pa_t c_i\|_{L^2}+(\|\pa_t u\|_H+\|\rho\na\Phi\|_{L^2})\|\na c_i\|_{L^2} +\|\na c_i\|_{L^2}\|\na \Phi\|_{L^\infty}+\|c_i\|_{L^4}\|\rho\|_{L^4}\\
\lesssim& \|\pa_t c_i\|_{L^2}+\|\pa_t u\|_H+1, \quad t\ge T_2
\eal
\ee
where in the second line, we used \eqref{ui}, and in the last line, we used the fact that all the terms, other than $\|\pa_t c_i\|_{L^2}$ and $\|\pa_t u\|_H$, are controlled by \eqref{Rq}, for times $t\ge T_2$.

Then, from the embedding $H^2\hookrightarrow L^\infty$, we have that
\be
\|c_i\|_{L^\infty}\lesssim\|\D c_i\|_{L^2}+1\le R_{10}(\|\pa_t c_i\|_{L^2}+\|\pa_t u\|_H+1),\quad t\ge T_2\la{R'}
\ee
where $R_{10}$ depends only on parameters and boundary conditions. Thus, returning to \eqref{ptc1}, we obtain
\be\la{ptc}
\bal
\fr{d}{dt}\|\pa_t c_i\|_{L^2}^2\lesssim& R_{10}^2\|\pa_t u\|_H^2(\|\pa_t c_i\|_{L^2}^2 +\|\pa_t u\|_H^2)\\
&+(\|\pa_t c_1\|_{L^2}^2+\|\pa_t c_2\|_{L^2}^2+\|\pa_t u\|_H^2)(R_5+R_{10}^2+1),\quad t\ge T_2
\eal
\ee
where again we used \eqref{Rq} to bound $\|\na\Phi\|_{L^\infty}^2,\|c_i\|_{L^3}^2$. Finally, adding \eqref{ptu} to \eqref{ptc} summed in $i$, we have
\be
\bal
\fr{d}{dt}\left(\|\pa_t c_1\|_{L^2}^2+\|\pa_t c_2\|_{L^2}^2+\|\pa_t u\|_H^2\right)\le& R_{11}(\|\pa_t u\|_H^2+1)\left(\|\pa_t c_1\|_{L^2}^2+\|\pa_t c_2\|_{L^2}^2+\|\pa_t u\|_H^2\right),\quad t\ge T_2\la{this}
\eal
\ee
where $R_{11}$ depends only on parameters and boundary conditions. Now we fix $t\ge T_3=T_2+1$. Then by \eqref{bRu} and \eqref{bRc}, there exists a time $t_0\in(t-1,t)$ such that
\be
\|\pa_t c_1(t_0)\|_{L^2}^2+\|\pa_t c_2(t_0)\|_{L^2}^2+\|\pa_t u(t_0)\|_H^2\le  R_{12}= R_6+2R_9\la{bR}
\ee
where $R_{12}$ depends only on parameters and boundary conditions. Thus, integrating \eqref{this} using \eqref{bRu} and \eqref{bR}, we obtain
\be
\bal
&\|\pa_t c_1(t)\|_{L^2}^2+\|\pa_t c_2(t)\|_{L^2}^2+\|\pa_t u(t)\|_H^2\\
\le&\left(\|\pa_t c_1(t_0)\|_{L^2}^2+\|\pa_t c_2(t_0)\|_{L^2}^2+\|\pa_t u(t_0)\|_H^2\right)\exp\left(R_{11}\int_{t_0}^t\|\pa_su(s)\|_H^2+1\,ds\right)\\
\le& R_{13}=R_{12}e^{R_{11}(R_{6}+1)},\quad t\ge T_3=T_2+1.\la{ptcu}
\eal
\ee

\textbf{Step 6. Absorbing ball for $\|c_i\|_{H^2}$ and $\|u\|_{H^2}$.} Finally, since by \eqref{ui}, \eqref{Rq} and \eqref{Dc} we have, 
\be
\bal
\|Au\|_{H}+\|\D c_1\|_{L^2}+\|\D c_2\|_{L^2}\lesssim&\|\pa_t u\|_H+\|\pa_t c_1\|_{L^2}+\|\pa_t c_2\|_{L^2}+1,\quad t\ge T_2,
\eal
\ee
it follows from \eqref{ptcu} that
\be
\|Au\|_{H}+\|\D c_1\|_{L^2}+\|\D c_2\|_{L^2}\le R^*,\quad t\ge T_3
\ee
for some constant $R^*$ depending only on parameters and boundary conditions. We recall that $T_3=T_2+1=T_1+2$, where $T_1$ is from Lemma \ref{balllem}. Thus $T_3$ depends only on the $L^2$ norms of the initial data, $\|w_0\|_{\mathcal H}$. We take $T_0$ from the statement of the proposition to be equal to $T_3$. This completes the proof of \eqref{ab}.

We emphasize that \eqref{ab} establishes \textit{long time} bounds that are \textit{independent of initial conditions}. However, for the proof of \eqref{ab1} (and of later results), it is helpful to have quantitative bounds on NPS solutions that hold for \textit{all} time. Such bounds are obtainable at the expense of introducing dependencies on initial data, as we show in the following lemma:
\begin{lemma}\la{l2!}
    Let $w=(c_1,c_2,u)$ be a strong solution of NPS \eqref{np}-\eqref{stokes} with boundary conditions \eqref{gamma}-\eqref{noslip}. Then there exists a constant $\mathcal G$ dependingly only on parameters and boundary conditions and a constant $\mF_0$ depending additionally on initial conditions (specifically $\|w_0\|_\mV$) such that for all $t\ge 0$,
    \be
    \int_0^t\|w(s)\|_{\mH^2}^2\le \mF_0+\mathcal G t.
    \ee
\end{lemma}
\begin{proof}
Taking time $T_0$ as determined by part \eqref{ab} of Proposition \ref{p1}, we have for all $t\ge T_0$,
\be\la{dlrj}
\int_{T_0}^t\|w(s)\|_{\mH^2}^2\,ds\le R^2(t-T_0)\le R^2 t.
\ee
Now consider $0\le t<T_0$. Then integrating \eqref{ddtu} and using \eqref{FB} and \eqref{fint}, we obtain
\be\la{dlrjteh}
\bal
\|u(t)\|_V^2+\nu\int_0^t\|Au(s)\|_H^2\,ds\le& \|u(0)\|_V^2+C(\mathcal F(0)+1)(\mathcal{F}(0)+t)\\
\le& \|u(0)\|_V^2+C(\mathcal F(0)+1)(\mathcal{F}(0)+T_0)={\tilde C}_0 \text{ for all } 0\le t<T_0.
\eal
\ee
Now integrating \eqref{ddtq2} using \eqref{FB}, \eqref{fint} and \eqref{dlrjteh}, we have for $i=1,2$ and for all $0\le t<T_0$,
\be
\bal
D_i\int_0^t\|\D c_i(s)\|_{L^2}^2\,ds\le& \|\na c_i(0)\|_{L^2}^2\exp\left(C(t+{\tilde C}_0^2+(\mathcal{F}(0)+1)^2t+(\mathcal{F}(0)+t))\right)\\
\le&\|\na c_i(0)\|_{L^2}^2\exp\left(C(T_0+{\tilde C}_0^2+(\mathcal{F}(0)+1)^2T_0+(\mathcal{F}(0)+T_0))\right).
\eal
\ee
Combining this inequality with \eqref{dlrjteh} and \eqref{dlrj} gives us the conclusion of the lemma.
\end{proof}

Now we prove part \eqref{ab1} of Proposition \ref{p1}. Suppose $w=(c_1,c_2,u)$ and $\bar w=(\bar c_1,\bar c_2,\bar u)$ are two solutions of NPS with boundary conditions \eqref{gamma}-\eqref{noslip}. Then $\tilde w=\bar w- w=(\bar c_1- c_1,\bar c_2-c_2,\bar u- u)$ satisfies
\begin{align}
    \pa_t \tilde c_i-D_i\D\tilde c_i=&D_iz_i\na\cdot(c_i\na\tilde\Phi+\tilde c_i\na\bar\Phi)-u\cdot\na\tilde c_i-\tilde u\cdot\na\bar c_i,\quad i=1,2 \la{npdiff}\\
    \pa_t\tilde u +\nu A\tilde u=&-K\mathbb{P}(\rho\na\tilde\Phi)-K\mathbb{P}(\tilde\rho\na\bar\Phi).\la{nsdiff}
\end{align}
Multiplying \eqref{npdiff} by $-\D\tilde c_i$ and integrating by parts, we obtain
\be\la{F1}
\bal
\fr{1}{2}\fr{d}{dt}\|\na\tilde c_i\|_{L^2}^2+D_i\|\D \tilde c_i\|_{L^2}^2=&-\int_\Omega(D_iz_i\na\cdot(c_i\na\tilde\Phi+\tilde c_i\na\bar\Phi)-u\cdot\na\tilde c_i-\tilde u\cdot\na\bar c_i)\D\tilde c_i\,dx
\eal
\ee
from which it follows from Hölder and Young's inequalities,
\be\la{diffc}
\bal
\fr{d}{dt}\|\na\tilde c_i\|_{L^2}^2+D_i\|\D \tilde c_i\|_{L^2}^2\lesssim& \|\na c_i\|_{L^2}^2\|\na\tilde\Phi\|_{L^\infty}^2+\|c_i\|_{L^4}^2\|\tilde \rho\|_{L^4}^2+\|\na\tilde c_i\|_{L^2}^2\|\na\bar\Phi\|_{L^\infty}^2+\|\tilde c_i\|_{L^4}^2\|\bar\rho\|_{L^4}^2\\
&+\|u\|_{L^\infty}^2\|\na \tilde c_i\|_{L^2}^2+\|\tilde u\|_{L^6}^2\|\na\bar c_i\|_{L^3}^2\\
\le& G_i(t)(\sum_j\|\na \tilde c_j\|_{L^2}^2+\|\tilde u\|_V^2)
\eal
\ee
where 
\be\la{Gi}
G_i(t)=C(\|\na c_i\|_{L^2}^2+\|c_i\|_{L^4}^2+\|\na\bar\Phi\|_{L^\infty}^2+\|\bar\rho\|_{L^4}^2+\|u\|_{L^\infty}^2+\|\na\bar c_i\|_{L^3}^2).
\ee
Taking the inner product of \eqref{nsdiff} with $A\tilde u$ and integrating by parts, we obtain
\be\la{F2}
\bal
\fr{1}{2}\fr{d}{dt}\|\tilde u\|_{V}^2+\nu\|A\tilde u\|_H^2=-K\int_\Omega \mathbb{P}(\rho\na\tilde\Phi+\tilde\rho\na\bar\Phi)A\tilde u\,dx
\eal
\ee
from which it follows from Hölder and Young's inequalities,
\be\la{diffu}
\bal
\fr{d}{dt}\|\tilde u\|_V^2+\nu\|A\tilde u\|_{H}^2\lesssim& \|\rho\|_{L^2}^2\|\na\tilde\Phi\|_{L^\infty}^2+\|\tilde\rho\|_{L^2}^2\|\na\bar\Phi\|_{L^\infty}^2\\
\le&G_u(t)\sum_j\|\na\tilde c_j\|_{L^2}^2
\eal
\ee
where 
\be\la{Gu}
G_u(t)=C(\|\rho\|_{L^2}^2+\|\na\bar\Phi\|_{L^\infty}^2).
\ee
Summing \eqref{diffc} in $i$ and adding it to \eqref{diffu}, we obtain
\be
\bal
&\fr{d}{dt}\left(\sum_i\|\na\tilde c_i\|_{L^2}^2+\|\tilde u\|_V^2\right)+\sum_iD_i\|\D\tilde c_i\|_{L^2}^2+\nu\|A\tilde u\|_H^2\\
\le& (G_1(t)+G_2(t)+G_u(t))\left(\sum_i\|\na\tilde c_i\|_{L^2}^2+\|\tilde u\|_V^2\right)
\eal
\ee
from which it follows that
\be\la{cts}
\bal
&\sum_i\|\na\tilde c_i(t)\|_{L^2}^2+\|\tilde u(t)\|_V^2+\int_0^t\left(\sum_iD_i\|\D\tilde c_i(s)\|_{L^2}^2+\nu\|A\tilde u(s)\|_H^2\right)\,ds\\
\le& \left(\sum_i\|\na\tilde c_i(0)\|_{L^2}^2+\|\tilde u(0)\|_V^2\right)\exp\left(\int_0^t(G_1(s)+G_2(s)+G_u(s))\,ds\right).
\eal
\ee
Now, from Lemma \ref{l2!} and Sobolev embeddings, we know that $G_i(s)$ and $G_u(s)$ are integrable in time. Thus the exponential term in \eqref{cts} is finite for all $t\ge 0$. It follows from \eqref{cts} that, for each $t\ge 0$, $S(t):\mathcal{V}^+\to\mathcal{V}^+$ is continuous (in fact, it is Lipschitz continuous, with a time dependent Lipschitz constant that grows at most exponentially in time).

Finally we prove \eqref{ab2}, using the log-convexity method of Agmon and Nirenberg \cite{agmon} (see also \cite{abdo,fsqg}). As in the proof of \eqref{ab1}, we consider two solutions $w=(c_1,c_2,u), \bar w=(\bar c_1,\bar c_2, \bar u)$ of NPS with boundary conditions \eqref{gamma}-\eqref{noslip}, with differences denoted by a tilde. This time we additionally assume that $\tilde w(0)=w(0)-\bar w(0)\neq 0.$ Then, we consider the following quantities
\be\la{Y}
Y(t)=\log\left(\fr{1}{E_0(t)}\right),\quad E_0(t)=\|\tilde c_1(t)\|_{L^2}^2+\|\tilde c_2(t)\|_{L^2}^2+\|\tilde u(t)\|_{H}^2.
\ee
To prove the injectivity of the solution map $S$, it suffices to show that for all $t\ge 0$, we have $E_0(t)> 0$. That is, it suffices to show that $Y(t)$ remains finite for all time. 

Multiplying \eqref{npdiff} by $\tilde c_i$ and integrating by parts, we see that 
\be\la{LI}
\bal
\fr{1}{2}\fr{d}{dt}\|\tilde c_i\|_{L^2}^2&=-D_i\|\na\tilde c_i\|_{L^2}^2+\left( D_iz_i\na \cdot(c_i\na\tilde\Phi+\tilde c_i\na\bar\Phi)-u\cdot\na\tilde c_i-\tilde u\cdot\na\bar c_i,\tilde c_i\right)_{L^2}\\
&=:-D_i\|\na \tilde c_i\|_{L^2}^2+(L_i,\tilde c_i)_{L^2}.
\eal
\ee
And, taking the inner product of \eqref{nsdiff} with $\tilde u$ and integrating by parts, we obtain
\be\la{LU}
\bal
\fr{1}{2}\fr{d}{dt}\|\tilde u\|_{H}^2&=-\nu\|\tilde u\|_V^2+(-K\mathbb{P}(\rho\na\tilde\Phi+\tilde\rho\na\bar\Phi),\tilde u)_H\\
&=:-\nu\|\tilde u\|_V^2+(L_u,\tilde u)_H
\eal
\ee
Thus, the function $Y(t)$ satisfies the following differential equation
\be\la{dY}
\fr{1}{2}\fr{d}{dt}Y(t)=-\fr{1}{2}\fr{1}{E_0}\fr{d}{dt}E_0=\fr{E_1}{E_0}-\fr{\sum_i(L_i,\tilde c_i)_{L^2}+(L_u,\tilde u)_{H}}{E_0}
\ee
where 
\be\la{E1}
E_1=D_1\|\na\tilde c_1\|_{L^2}^2+D_2\|\na\tilde c_2\|_{L^2}^2+\nu\|\tilde u\|_V^2.
\ee
Estimating $(L_i,\tilde c_i)_{L^2}$ using Hölder, Young's, and elliptic inequalities, and the fact that $(u\cdot\na \tilde c_i,\tilde c_i)_{L^2}=0$, we obtain
\be\la{li}
\bal
|(L_i,\tilde c_i)_{L^2}|\lesssim& \|\tilde c_i\|_{L^2}(\|\na c_i\|_{L^4}\|\na\tilde\Phi\|_{L^4}+\|c_i\|_{L^\infty}\|\tilde\rho\|_{L^2}+\|\na\tilde c_i\|_{L^2}\|\na\bar\Phi\|_{L^\infty}+\|\tilde c_i\|_{L^2}\|\bar\rho\|_{L^\infty}\\
&+\|\tilde u\|_{L^4}\|\na\bar c_i\|_{L^4})\\
\le&C(\|\na c_i\|_{L^4}+\|c_i\|_{L^\infty}+\|\na\bar\Phi\|_{L^\infty}^2+\|\bar\rho\|_{L^\infty}+\|\na \bar c_i\|_{L^4}^2)(\|\tilde c_1\|_{L^2}^2+\|\tilde c_2\|_{L^2}^2)\\
&+D_i\|\na \tilde c_i\|_{L^2}^2+\nu\|\tilde u\|_V^2\\
\le&B_i(t)E_0+E_1
\eal
\ee
where 
\be
B_i(t)=C(\|\na c_i\|_{L^4}+\|c_i\|_{L^\infty}+\|\na\bar\Phi\|_{L^\infty}^2+\|\bar\rho\|_{L^\infty}+\|\na \bar c_i\|_{L^4}^2).
\ee
Similarly, we have
\be\la{lu}
\bal
|(L_u,\tilde u)_H|\lesssim&\|\tilde u\|_H(\|\rho\|_{L^\infty}\|\na\tilde\Phi\|_{L^2}+\|\tilde\rho\|_{L^2}\|\na\bar\Phi\|_{L^\infty})\\
\lesssim& (\|\rho\|_{L^\infty}+\|\na\bar\Phi\|_{L^\infty})(\|\tilde c_1\|_{L^2}^2+\|\tilde c_2\|_{L^2}^2+\|\tilde u\|_H^2)\\
\le&B_u(t) E_0
\eal
\ee
where
\be
B_u(t)=C(\|\rho\|_{L^\infty}+\|\na\bar\Phi\|_{L^\infty}).
\ee
It now follows from \eqref{dY} that
\be\la{YY}
\fr{1}{2}\fr{d}{dt}Y(t)\le 3\fr{E_1}{E_0}+B_1(t)+B_2(t)+B_u(t)=:3\fr{E_1}{E_0}+B(t)
\ee
where $B(t)$ is integrable in time due to Lemma \ref{l2!} and Sobolev embeddings.

Now, in order to control $Y(t)$, it suffices to control the growth of the Dirichlet quotient $\fr{E_1}{E_0}$. A direct computation shows that the Dirichlet quotient satisfies the following differential equation,
\be\la{dE1E0}
\fr{d}{dt}\fr{E_1}{E_0}=\fr{1}{E_0}\fr{d}{dt}E_1+\fr{E_1}{E_0}\fr{d}{dt}Y.
\ee
It follows from \eqref{F1} and \eqref{F2} that $E_1$ satisfies
\be\la{dE1}
\fr{1}{2}\fr{d}{dt}E_1+E_2=\sum_i -(L_i,D_i\D\tilde c_i)_{L^2}+(L_u,\nu A\tilde u)_H
\ee
where 
\be
E_2=D_1^2\|\D\tilde c_1\|_{L^2}^2+D_2^2\|\D\tilde c_2\|_{L^2}+\nu^2\|A\tilde u\|_H^2
\ee
and $L_i$ and $L_u$ are defined in \eqref{LI} and \eqref{LU}, respectively. Now, substituting \eqref{dY} and \eqref{dE1} into \eqref{dE1E0}, we obtain
\be\la{dede}
\fr{1}{2}\fr{d}{dt}\fr{E_1}{E_0}=-\fr{E_2}{E_0}-\fr{1}{E_0}\sum_i(L_i,D_i\D\tilde c_i)_{L^2}+\fr{1}{E_0}(L_u,\nu A\tilde u)_H+\fr{E_1}{E_0}\left(\fr{E_1}{E_0}-\fr{\sum_i(L_i,\tilde c_i)_{L^2}+(L_u,\tilde u)_{H}}{E_0}\right).
\ee
A direct calculation shows that
\be\bal
-\fr{E_2}{E_0}+\fr{E_1^2}{E_0^2}&=-\left\|\left(-D_1\D-\fr{E_1}{E_0}\right)\fr{\tilde c_1}{E_0^\fr{1}{2}}\right\|_{L^2}^2-\left\|\left(-D_2\D-\fr{E_1}{E_0}\right)\fr{\tilde c_2}{E_0^\fr{1}{2}}\right\|_{L^2}^2-\left\|\left(\nu A-\fr{E_1}{E_0}\right)\fr{\tilde u}{E_0^\fr{1}{2}}\right\|_H^2\\
&=:-\mathbb D\le 0.
\eal
\ee
Thus \eqref{dede} may be rewritten
\be
\bal
&\fr{1}{2}\fr{d}{dt}\fr{E_1}{E_0}+\mathbb D\\
=&\sum_i\fr{1}{E_0^\fr{1}{2}}\left(L_i,\left(-D_i\D-\fr{E_1}{E_0}\right)\fr{\tilde c_i}{E_0^\fr{1}{2}}\right)_{L^2}+\fr{1}{E_0^\fr{1}{2}}\left(L_u,\left(\nu A-\fr{E_1}{E_0}\right)\fr{\tilde u}{E_0^\fr{1}{2}}\right)_H
\eal
\ee
and it follows from Cauchy-Schwarz and Young's inequalities that
\be
\fr{d}{dt}\fr{E_1}{E_0}+\mathbb D\le \fr{1}{E_0}\sum_i\|L_i\|_{L^2}^2+\fr{1}{E_0}\|L_u\|_H^2.
\ee
Then, estimating $\|L_i\|_{L^2}^2$ and $\|L_u\|_H^2$ as in \eqref{li} and \eqref{lu}, we find that
\begin{align}
    \|L_i\|_{L^2}^2\le& \alpha_i(t)E_0+\beta_i(t)E_1\\
    \|L_u\|_{L^2}^2\le& \alpha_u(t)E_0
\end{align}
where
\begin{align}
    \alpha_i(t)=&C(\|\na c_i\|_{L^4}^2+\|c_i\|_{L^\infty}^2+\|\bar\rho\|_{L^\infty})\\
    \beta_i(t)=&C(\|\na\bar\Phi\|_{L^\infty}^2+\|\na \bar c_i\|_{L^4}^2)\\
    \alpha_u(t)=&C(\|\rho\|_{L^\infty}^2+\|\na\bar\Phi\|_{L^\infty}^2)
\end{align}
which are all integrable in time due to Lemma \ref{l2!} and Sobolev embeddings. Thus, $\fr{E_1}{E_0}$ satisfies
\be
\fr{d}{dt}\fr{E_1}{E_0}\le\sum_i\alpha_i(t)+\alpha_u(t)+\fr{E_1}{E_0}\sum_i\beta_i(t)
\ee
and so
\be
\fr{E_1(t)}{E_0(t)}\le \left(\int_0^t\sum_i\alpha_i(s)+\alpha_u(s)\,ds+\fr{E_1(0)}{E_0(0)}\right)\exp\left(\int_0^t\sum_i\beta_i(s)\,ds\right)=E(t),\quad t\ge 0.
\ee
Returning to \eqref{YY} and integrating, we obtain
\be
\fr{1}{2}Y(t)\le \fr{1}{2}Y(0)+\int_0^t\left(3E(s)+B(s)\right)\,ds.
\ee
Thus $Y(t)$ is finite for all finite $t$, and the proof of \eqref{ab2} is complete. This completes the proof of the proposition.
\end{proof}

A consequence of (\ref{ab}) of the preceding proposition is the following corollary.
\begin{cor}\la{brbr}
For $R$ and $B_R^2$ determined by Proposition \ref{p1} (\ref{ab}), there exists a time $T_R>0$ such that $S(t)B_R^2\subset B_R^2$ for all $t\ge T_R$.
\end{cor}
\begin{proof}
This follows from Proposition \ref{p1} and Remark \ref{monotone}. Indeed, since $B_R^2$ is bounded in $\mathcal{H}^2$ and thus also in $\mathcal H$, it follows that by choosing $T_R$ large enough (relative to $R$), we have $S(t)B_R^2\subset B_R^2$ for all $t\ge T_R$.
\end{proof}

The properties of $S(t)$ established in Proposition \ref{p1} are sufficient to prove the existence of a global attractor, following the same line of reasoning as in \cite{cf}. For the sake of completeness, we provide a proof of Theorem \ref{globa} below.
\begin{proof}
Recalling the bounded set $B_R^2\subset\mathcal{H}^2$ obtained in Proposition \ref{p1} (\ref{ab}), we claim that
\be
\mA=\bigcap\limits_{t>0}S(t)B_R^2
\ee
is a global attractor for the NPS system. 

The compactness of $\mA$ in $\mathcal V$ follows from the continuity of $S(t):\mathcal V^+\to \mathcal V^+$ and the fact that $B_R^2$ is bounded in $\mH^2$ and thus compact in $\mV$.

We now prove the invariance of $\mathcal A$. Let us fix $t\ge 0$. To show $S(t)\mA\subset\mA$, we first take $x\in \mathcal{A}$. Then for any $\tau> 0$, there exists $y=y(\tau)\in B_R^2$ such that $x=S(\tau)y$. So $S(t)x=S(t)S(\tau)y=S(t+\tau)y$, and thus $S(t)x\in\bigcap_{s> t}S(s)B_R^2.$ It remains to verify that $S(t)x\in S(s)B_R^2$ for all $s\le t$. To show this, fix $s\le t$ and note that $S(t)x=S(s)S(t-s)x$, and so it suffices to show that $S(t-s)x\in B_R^2.$ Since $x\in \mathcal{A}$, there exists $y'\in B_R^2$ such that $x=S(T_R+s)y'$ where $T_R$ is obtained from Corollary \ref{brbr}. Then, $S(t-s)x=S(t-s+T_R+s)y'=S(T_R+t)y'\in B_R^2$, where the last inclusion follows from Corollary \ref{brbr}. This completes the proof of $S(t)\mathcal A\subset \mathcal A$. To prove the opposite inclusion, we fix $x\in\mathcal{A}.$ Then there exists $y\in B_R^2$ such that $x=S(t)y.$ Then fixing $s>0$, we similarly have $z\in B_R^2$ such that $x=S(t+s)z$. Thus $S(t)y=S(t+s)z=S(t)S(s)z$, and from the injectivity of $S(t)$, it follows that $y=S(s)z$ and thus $y\in S(s)B_R^2$. Since $s>0$ was arbitrary, we conclude that in fact $y\in \mathcal{A}.$ Therefore $x=S(t)y\in S(t)\mathcal{A}$. Thus we have shown $\mathcal{A}\subset S(t)\mathcal{A}$ and the proof of the invariance of $\mathcal{A}$ is complete.

We now prove the maximality of $\mA$. Suppose that $\mB\subset \mV^+$ is bounded and satisfies $S(t)\mB=\mB$ for all $t\ge 0$. Since $\mB$ is bounded in $\mV$, by the same reasoning used to prove Corollary \ref{brbr}, there exists $t_*$ such that $S(t)\mB\subset B_R^2$ for all $t\ge t_*$. Now we fix $x\in \mB$. We aim to show that $x\in \mA$. To this end, our goal is to show that for every $s>0$, there exists $y=y(s)\in B_R^2$ such that $x=S(s)y$. Because $\mB$ is invariant under $S$ by assumption, it follows that, for every $s>0$, there exists $z=z(s)\in \mB$ such that $x=S(s+t_*)z$. Thus $x=S(s)S(t_*)z$, and since $S(t_*)z\in B_R^2$, we have shown $x\in S(s)B_R^2$, and thus $x\in\mathcal A.$

To prove the attractor property (Definition \ref{GA} (\ref{att})), it suffices to show that for every $v\in\mV$, the omega set of $v$
\be
\omega(v)=\{u\in \mV\,|\,\text{there exists a sequence of times } t_n\to \infty \text{ such that } \lim_{n\to\infty}\|u-S(t_n)v\|_\mV=0\}
\ee
is bounded in $\mV$ and invariant under $S(t)$ for all $t\ge 0$. Indeed, if this is the case, then Definition \ref{GA} (\ref{maximal}) applies to $\mB=\omega(v)$ so that $\omega(v)\subset\mA$. Then, for the sake of contradiction, if we assume that there exists $\delta>0$ and a sequence of times $t_n\to \infty$ such that $d_\mV(S(t_n)v,\mA)\ge \delta$, then for all $n$ large enough, we have $S(t_n)v\in B_R^2$ (Proposition \ref{p1} (\ref{ab})), and from the compactness of $B_R^2$ in $\mV$ it follows that there exists $\bar v\in \mV^+$ and a subsequence $t_{n_j}$ such that $S(t_{n_j})v\to \bar v$ in $\mV$. By definition, we have $\bar v\in\omega(v)$. But on the other hand, by continuity, we have $d_\mV(\bar v,\mA)\ge \delta$, which in particular implies $\bar v\not\in\mA\supset \omega(v)$. This contradiction implies $\lim\limits_{t\to\infty}d_\mV(S(t)v,\mA)=0$. 

It remains to prove the boundedness and invariance of $\omega(v)$. The boundedness follows from the fact that for all $u\in\omega(v)$, there exists a sequence $t_n\to\infty$ such that $\lim\limits_{n\to\infty}\|u-S(t_n)v\|_\mV=0$ and the fact that $S(t_n)v$ is in  $B_R^2$ (a bounded set in $\mV$) for all sufficiently large $n$. To show invariance, let us first take $u\in\omega(v)$, then there exists $t_n\to\infty$ such that $u=\lim\limits_{n\to\infty}S(t_n)v$ so that, for any $s\ge 0$, we have $S(s)u=\lim\limits_{n\to\infty}S(s+t_n)v.$ This shows that $S(s)\omega(v)\subset\omega(v).$ To prove the opposite inclusion, we fix $s\ge 0$, and we take $u\in \omega(v)$ and a sequence $t_n\to\infty$ such that $u=\lim\limits_{n\to\infty}S(t_n)v$. Then we consider the sequence $S(t_n-s)v$ where we assume that $n$ is large enough so that $t_n-s\ge 0$. Since $S(t_n-s)v\in B_R^2$ for all sufficiently large $n$ and because $B_R^2$ is compact in $\mV$, there exists a subsequence $t_{n_j}\to\infty$ such that $S(t_{n_j}-s)v$ converges to some $\bar v\in\omega(v)$. Thus, we have that $S(s)S(t_{n_j}-s)v=S(t_{n_j})v$ converges to both $S(s)\bar v$ and $u$. By the uniqueness of limits, we have $u=S(s)\bar v$, and so $u\in S(s)w(v)$. This establishes $\omega(v)\subset S(s)\omega(v)$, and completes the proof of the invariance of $\omega(v)$.

Lastly, we establish the connectedness of $\mA$. For the sake of contradiction, suppose there exist disjoint, nonempty open sets $O_1$ and $O_2$ in $\mV$ such that $\mA\subset O_1\sqcup O_2$ and $\mA\cap O_1\neq \emptyset$, $\mA\cap O_2\neq \emptyset$. We fix $x_1\in\mA\cap O_1$, $x_2\in\mA\cap O_2$, and a time $t>0$. Then there exist $y_1=y_1(t)$ and $y_2=y_2(t)$ in $B_R^2$ such that $x_1=S(t)y_1$ and $x_2=S(t)y_2$. We denote by $\gamma=\gamma(t)\subset B_R^2$ the line connecting $y_1$ and $y_2$. Then by the injectivity and continuity of $S(t)$, $S(t)\gamma$ is a simple curve connecting $x_1$ and $x_2$. Thus, there exists $x(t)\in S(t)\gamma$ such that $x(t)\in F=\mV\setminus (O_1\sqcup O_2)$, and we let $y(t)\in\gamma$ be the preimage of $x(t)$ under $S(t)$ i.e. $x(t)=S(t)y(t)$. We observe that $F$ is closed in $\mV$ and is disjoint from $\mA$.

Now, for all $t$ large enough, we have $x(t)=S(t)y(t)\in B_R^2$ (this inclusion follows from Corollary \ref{brbr} and the fact $y(t)\in\gamma\subset B_R^2$). Thus, by compactness of $B_R^2$, there exists $t_n\to\infty$ such that $x(t_n)\to_\mV \bar x$ for some $\bar x\in\mV^+$. Since each $x(t_n)\in F$ and $F$ is closed in $\mV$, we have that $\bar x\in F$. Now, we reach a contradiction if we show that $\bar x$ is in $\mA$, which is disjoint from $F$. To this end, fix $s>0$, and consider the sequence $S(t_n-s)y(t_n)$ for sufficiently large $n$ so that $t_n-s\ge 0$. Since $y(t_n)\in B_R^2$, we have $S(t_n-s)y(t_n)\in B_R^2$ for all $n$ sufficiently large, and by compactness there is a subsequence $t_{n_j}$ so that $S(t_{n_j}-s)y(t_{n_j})$ converges strongly in $\mV$ to some $\bar y\in B_R^2$. Then, on one hand $S(s)S(t_{n_j}-s)y(t_{n_j})$ converges strongly in $\mV$ to $S(s)\bar y$. On the other hand, since $S(s)S(t_{n_j}-s)y(t_{n_j})=S(t_{n_j})y(t_{n_j})=x(t_{n_j})$, the same sequence converges to $\bar x$. Thus $\bar x=S(s)\bar y$. Since $\bar y\in B_R^2$ and $s>0$ was arbitrary, we have shown $\bar x\in \mA$. This contradiction completes the proof of connectedness, and the proof of Theorem \ref{globa} is complete.
\end{proof}

\subsection{Proof of Theorem \ref{fdf}.}
In this subsection, we show that the global attractor $\mathcal{A}$ obtained in the previous subsection has finite fractal (box-counting) dimension. We recall the definition of the fractal dimension. For a compact set $K\subset \mV$, the fractal dimension of $K$ is given by
\be\la{fd}
d_f(K)=\limsup_{r\to 0^+}\fr{\log n_K(r)}{\log(1/r)}
\ee
where 
\be\la{nK}
n_K(r)=\text{minimum number of balls in $\mV$ of radii at most $r$ needed to cover $K$.}
\ee
An equivalent definition is given by
\be\la{fddef}
d_f(K)=\inf\{D>0\,|\, \limsup_{r\to 0^+} r^D n_K(r)\}.
\ee

As a first step in estimating $d_f(\mA)$, we consider, for initial conditions $w_0\in\mV^+$, the linearization of NPS \eqref{np}-\eqref{stokes} around the solution $w(t)=S(t)w_0$, which is obtained by taking initial conditions $\bar w_0\in\mV^+$ close to $w_0$, considering the time evolution equations satisfied by the perturbation $\tilde w(t)=(\tilde c_1(t),\tilde c_2(t),\tilde u(t))=S(t)\bar w_0-S(t) w_0$ (c.f. \eqref{npdiff},\eqref{nsdiff}), and dropping the terms that are quadratic in $\tilde w$:
\begin{align}
    \pa_t \underline c_i=&D_i\D \underbar c_i+D_iz_i\na\cdot(c_i\na\underline\Phi+\underline c_i\na\Phi)-u\cdot\na\underline c_i-\underline u\cdot\na c_i=:D_i\D\underline c_i+\mL_i(t;w_0) \underline w\la{LS1}\\
    \underline\Phi=&\fr{1}{\ep}(-\D_D)^{-1}\underline\rho=\fr{1}{\ep}(-\D_D)^{-1}(\underline c_1-\underline c_2)\la{tpois}\\
    \pa_t \underline u=&-\nu A\underline u-K\mathbb{P}(\rho\na\underline\Phi)-K\mathbb{P}(\underline\rho\na\Phi)=:-\nu A\underline u+\mL_u(t;w_0)\underline w\la{LS2}
\end{align}
where $-\D_D$ is the homogeneous Dirichlet Laplace operator, and we view $\mL_i(t;w_0),\mL_u(t;w_0)$ as time dependent, first order linear operators with \textit{known} coefficients given in terms of $w(t)$, which in turn is uniquely determined by $w_0$ via the solution operator $S(t)$. Above, we have replaced tildes with underlines to highlight the fact that $\underline w$ satisfies a linearized version of the full, nonlinear system satisfied by the actual perturbation $\tilde w$. 

We denote by 
\be S'(t;w_0):\mV_0\to \mV_0\ee
the solution map to the linear system \eqref{LS1}-\eqref{LS2} such that, for $\underline w_0\in \mV_0$, $S'(t;w_0)\underline w_0=(\underline c_1,\underline c_2,\underline u)$ is the unique solution to \eqref{LS1}-\eqref{LS2} with initial conditions $\underline w_0=(\underline c_1(0),\underline c_2(0),\underline u(0))$ and homogeneous Dirichlet boundary conditions for $\underline c_i$ and $\underline u$. 

The linearized dynamics of small perturbations $\underline w_0\in\mVo$, described by $S'(t,w_0)\underline w_0$, should approximate the nonlinear dynamics of these perturbations, $\tilde w(t)=S(t)(w_0+\underline w_0)-S(t)w_0$. Below, in Proposition \ref{ctsdiff}, we make this statement precise: 

\begin{prop}\la{ctsdiff}
Let $w_0\in\mA$ and $0<r\le 1$. Then there exists a constant $\mF_\mA$ (independent of $r$) such that for every $\bar w_0\in\mV_\gamma^+= \{(f,g,h)\in \mV^+\,|\, f_{|\pa\Omega}=\gamma_1, g_{|\pa\Omega}=\gamma_2\}$ with $0<\|\bar w_0-w_0\|_\mV\le r$, we have
\be\la{FAA}
\|S(t)\bar w_0-S(t) w_0-S'(t,w_0)(\bar w_0-w_0)\|_{\mV}\le e^{\mF_\mA(1+t)} \|\bar w_0-w_0\|_\mV^2.
\ee
Here, the constant $\mF_\mA$ depends on parameters, boundary conditions and on the diameter of $\mA$ in $\mV$ (since $\mA\subset B_R^2$, we may view $\mF_\mA$ as a function, say, of $R$). In particular, we have for each $w_0\in\mA$ and $t>0$,
\be
\lim_{r\to 0^+}\sup_{\substack{\bar w_0\in\mV_\gamma^+\\0< \|\bar w_0-w_0\|_\mV\le r}}\fr{\|S(t)\bar w_0-S(t) w_0-S'(t,w_0)(\bar w_0-w_0)\|_\mV}{\|\bar w_0-w_0\|_\mV}=0.
\ee
\end{prop}

\begin{proof}
We fix $w_0\in\mA$ and $\bar w_0\in\mV_\gamma^+$ as in the hypothesis. We denote 
\begin{align*}
\tilde w_0=&\bar w_0-w_0\\
S(t)w_0=&w(t)=(c_1(t),c_2(t),u(t))\\
S(t)\bar w_0=&\bar w(t)=(\bar c_1(t),\bar c_2(t),\bar u(t))\\
S'(t,w_0)\tilde w_0=&\underline w(t)=(\underline c_1,\underline c_2,\underline u)\\
\xi_i=&\bar c_i-c_i-\underline c_i\\
\xi_\rho=&\bar\rho-\rho-\underline \rho\\
\xi_\Phi=&\bar\Phi-\Phi-\underline \Phi\\
\xi_u=&\bar u-u-\underline u\\
\tilde w=&(\tilde c_1,\tilde c_2,\tilde u)=(\bar c_1-c_1,\bar c_2-c_2,\bar u-u)
\end{align*}
where $\underline\rho=\underline c_1-\underline c_2$ and $\underline\Phi=(-\ep\D_D)^{-1}\underline\rho$. For each $t>0$, the function $S(t)\bar w_0-S(t) w_0-S'(t,w_0)\tilde w_0$ has zero trace, so by Poincaré's inequality, we have
\be
\bal
N_1+N_2+N_u:=&\|\na \xi_1\|_{L^2}^2+\|\na \xi_2\|_{L^2}^2+\|\xi_u\|_V^2\\
\gtrsim& \|S(t)\bar w_0-S(t) w_0-S'(t,w_0)\tilde w_0\|_\mV^2.
\eal
\ee
So, it suffices to control the growth of $N_1, N_2, N_u$. We achieve this by computing the time derivatives of $N_1,N_2,N_u$. We start with $N_i$, $i=1,2$.

We observe that $\xi_i$ satisfies the equation
\be
\pa_t \xi_i-D_i\D\xi_i=D_iz_i\na\cdot(\xi_i\na\Phi+c_i\na\xi_\Phi)-\xi_u\cdot\na c_i-u\cdot\na\xi_i+D_iz_i\na\cdot(\tilde c_i\na\tilde\Phi)-\tilde u\cdot\na\tilde c_i.
\ee
Multiplying the above equation by $-\D \xi_i$, integrating by parts, and using Hölder and Young's inequalities together with elliptic estimates, we obtain
\be\la{Ni}
\bal
\fr{dN_i}{dt}+D_i\|\D\xi_i\|_{L^2}^2\lesssim& \|\na\xi_i\|_{L^2}^2\|\na\Phi\|_{L^\infty}^2+\|\xi_i\|_{L^2}^2\|\rho\|_{L^\infty}^2+\|\na c_i\|_{L^2}^2\|\na\xi_\Phi\|_{L^\infty}^2+\|c_i\|_{L^4}^2\|\xi_\rho\|_{L^4}^2\\
&+\|\xi_u\|_{L^4}^2\|\na c_i\|_{L^4}^2+\|u\|_{L^\infty}^2\|\na\xi_i\|_{L^2}^2\\
&+\|\na\tilde c_i\|_{L^2}^2\|\na\tilde\Phi\|_{L^\infty}^2+\|\tilde c_i\|_{L^4}^2\|\tilde \rho\|_{L^4}^2+\|\tilde u\|_{L^\infty}^2\|\na\tilde c_i\|_{L^2}^2\\
\lesssim&\Xi_i(t)(N_1+N_2+N_u)+\left(\sum_j\|\na \tilde c_j\|_{L^2}^2+\|A\tilde u\|_H^2\right)\sum_j\|\na\tilde c_j\|_{L^2}^2
\eal
\ee
where
\be\la{Xii}
\Xi_i(t)=\|\na\Phi\|_{L^\infty}^2+\|\rho\|_{L^\infty}^2+\|\na c_i\|_{L^2}^2+\|c_i\|_{L^4}^2+\|\na c_i\|_{L^4}^2+\|u\|_{L^\infty}^2.
\ee
Similarly, $\xi_u$ satisfies the equation
\be
\pa_t\xi_u+\nu A\xi_u=-K\mathbb{P}(\xi_\rho\na\Phi+\rho\na\xi_\Phi)-K\mathbb{P}(\tilde\rho\na\tilde\Phi)
\ee
so taking the inner product with $A\xi_u$, integrating by parts, and using Hölder and Young's inequalities together with elliptic estimates, we obtain
\be\la{Nu}
\bal
\fr{dN_u}{dt}+\nu\|A\xi_u\|_{H}^2\lesssim& \|\xi_\rho\|_{L^2}^2\|\na\Phi\|_{L^\infty}^2+\|\rho\|_{L^2}^2\|\na\xi_\Phi\|_{L^\infty}^2+\|\tilde\rho\|_{L^2}^2\|\na\tilde\Phi\|_{L^\infty}^2\\
\lesssim& \Xi_u(t)(N_1+N_2)+\left(\sum_j\|\na \tilde c_j\|_{L^2}^2\right)^2
\eal
\ee
where 
\be\la{Xiu}
\Xi_u(t)=\|\na\Phi\|_{L^\infty}^2+\|\rho\|_{L^2}^2.
\ee
Thus, adding \eqref{Ni} and \eqref{Nu} and applying Grönwall's inequality, we obtain
\be
\bal
N_1(t)+N_2(t)+N_u(t)\le& C\sup_{s\in[0,t]}\sum_i\|\na\tilde c_i(s)\|_{L^2}^2\int_0^t\left(\sum_i\|\na\tilde c_i(s)\|_{L^2}^2+\|A\tilde u(s)\|_H^2\right)\,ds\\
&\times\exp\left(C\int_0^t\sum_i\Xi_i(s)+\Xi_u(s)\,ds\right)
\eal
\ee
where we have used the fact that $N_1(0)=N_2(0)=N_u(t)=0$. Then using \eqref{cts}, we find that
\be\la{above}
N_1(t)+N_2(t)+N_u(t)\le \|\tilde w_0\|_{\mV}^4\exp\left(C\int_0^t\sum_i(\Xi_i(s)+G_i(s))+\Xi_u(s)+G_u(s)\,ds\right).
\ee
Now, we recall the definitions of $G_i,G_u,\Xi_i,\Xi_u$ (c.f. \eqref{Gi},\eqref{Gu},\eqref{Xii},\eqref{Xiu}). By Lemma \ref{l2!} and using Sobolev embeddings, we have
\be
N_1(t)+N_2(t)+N_u(t)\le \|\tilde w_0\|_\mV^4e^{\tilde \mF_\mA(1+t)}
\ee
where $\tilde\mF_\mA$ depends on the norms $\|w_0\|_\mV,\|\bar w_0\|_\mV$. However, these norms are bounded by $R$ and $R+1$, respectively, by hypothesis. Thus we have shown \eqref{FAA} and the proof is complete.
\end{proof}

Motivated by the preceding proposition, we analyze the linear operator $S'(t,w_0)$ in the proposition below.

\begin{prop}\la{S'}
Let $w_0\in\mV^+$. The operator $S'(t,w_0):\mV_0\to\mV_0$ satisfies the following:
\begin{enumerate}[(i)]
    \item $S'(t,w_0)$ is injective for each $t\ge 0$
    \item\la{S'c} $S'(t,w_0):\mV_0\to\mV_0$ is bounded for each $t>0$ and satisfies
    \be
    \|S'(t,w_0)\underline w_0\|_{\mVo}\le \|\underline w_0\|_{\mVo} e^{c(1+t)}
    \ee
    for constant $c>0$ that depends only on parameters, boundary conditions, and $\|w_0\|_\mV$
    \item\la{S'cc} $S'(t,w_0)$ maps $\mV_0$ into $\mV_0\cap\mH^2$ boundedly for each $t>0$. Thus $S'(t,w_0)$ is a compact operator on $\mV_0$
\end{enumerate}
\end{prop}
\begin{proof}
The proof of the proposition follows the same steps as in the proof of Proposition \ref{p1}, and so we just highlight the main steps. We first prove (\ref{S'c}). Multiplying \eqref{LS1} by $-\D\underline c_i$ and integrating by parts, we obtain
\be
\bal
\fr{d}{dt}\|\na\underline c_i\|_{L^2}^2+D_i\|\D\underline c_i\|_{L^2}^2\lesssim& \|\na c_i\|_{L^4}^2\|\na\underline\Phi\|_{L^4}^2+\|c_i\|_{L^4}^2\|\underline\rho\|_{L^4}^2+\|\na\underline c_i\|_{L^2}^2\|\na\Phi\|_{L^\infty}^2+\|\underline c_i\|_{L^4}^2\|\rho\|_{L^4}^2\\
&+\|u\|_{L^\infty}^2\|\na\underline c_i\|_{L^2}^2+\|\underline u\|_{L^4}^2\|\na c_i\|_{L^4}^2\\
\le&A_i(t)\left(\sum_j\|\na\underline c_j\|_{L^2}^2+\|\underline u\|_V^2\right)
\eal
\ee
where 
\be
A_i(t)=C(\|\na c_i\|_{L^4}^2+\|c_i\|_{L^4}^2+\|\na\Phi\|_{L^\infty}^2+\|\rho\|_{L^4}^2+\|u\|_{L^\infty}^2)
\ee
is integrable in time due to Lemma \ref{l2!}. Similarly, taking the inner product of \eqref{LS2} with $A\underline u$ and integrating by parts, we obtain
\be
\fr{d}{dt}\|\underline u\|_V^2+\nu\|A\underline u\|_H^2\le A_u(t)\sum_j\|\na \underline c_j\|_{L^2}^2
\ee
where 
\be
    A_u(t)=C(\|\rho\|_{L^4}^2+\|\na\Phi\|_{L^4}^2)
\ee
is also integrable in time due to Lemma \ref{l2!}. Thus, we have
\be
\bal
&\fr{d}{dt}\left(\sum_i\|\na\underline c_i\|_{L^2}^2+\|\underline u\|_V^2\right)+\sum_i D_i\|\D \underline c_i\|_{L^2}^2+\nu\|A\underline u\|_H^2\\
\le& (A_1(t)+A_2(t)+A_u(t))\left(\sum_i\|\na\underline c_i\|_{L^2}^2+\|\underline u\|_V^2\right)
\eal
\ee
from which it follows that for every $t$, $\underline w$ obeys
\be\la{uvb}
\bal
&\sum_i\|\na\underline c_i(t)\|_{L^2}^2+\|\underline u(t)\|_V^2+\int_0^t\sum_i D_i\|\D \underline c_i(s)\|_{L^2}^2+\nu\|A\underline u(s)\|_H^2\,ds\\
\le& \left(\sum_i\|\na\underline c_i(0)\|_{L^2}^2+\|\underline u(0)\|_V^2\right)\exp\left(C\int_0^t (A_1(s)+A_2(s)+A_u(s))\,ds\right)<\infty.
\eal
\ee
The above bound, together with Lemma \ref{l2!} and Sobolev embeddings, establishes (\ref{S'c}). Starting from this continuity estimate and following the same boostrapping procedure as in Steps 3-6 in the proof of Proposition \ref{p1}, we obtain bounds on $\|\underline w(t)\|_{\mH^2}$ in terms of $\|\underline w_0\|_\mV$. These bounds also depend on time $t$ but are finite for each fixed $t$. Thus, for each fixed $t$, $S'(t,w_0)$ maps bounded sets of $\mV_0$ into bounded sets of $\mV_0\cap \mH^2$. This proves (\ref{S'cc}). 

Equipped with the $L^2(0,t;\mH^2)$ bounds established in \eqref{uvb}, we follow the same steps as in the proof of \eqref{ab2} of Proposition \ref{p1} to prove the injectivity of $S'(t,w_0)$.
\end{proof}

In addition to the preceding propositions, a crucial ingredient in proving the finite dimensionality of the global attractor is showing that sufficiently high (but finite) dimensional volume elements of $\mVo$ decay exponentially in time when transported by the linear flow $S'(t;w_0)$, for any $w_0\in \mV^+$. We make this statement more precise. We denote by $\mVo^N=\bigwedge^N\mVo$ the $N$-th exterior product of $\mVo$ with inner product $(\cdot,\cdot)_{\mVo^N}$ defined by
\be
(v_1\wedge\cdots\wedge v_N,w_1\wedge\cdots\wedge w_N)_{\mVo^N}=\det \{(v_i,w_i)_{\mVo}\}_{ij}
\ee
where $\{(v_i,w_i)_{\mVo}\}_{ij}$ is the $N\times N$ matrix with $ij$-th component given by $(v_i,w_i)_{\mVo}$.

Now, fix initial conditions $w_0\in\mV^+$ and a set of small initial perturbations $\tilde w_{1,0},...,\tilde w_{N,0}\in\mVo$. Then, consider the volume of the parallelepiped spanned by $\tilde w_{1,0},...,\tilde w_{N,0}$
\be
V_N(0)=\|\tilde w_{1,0}\wedge\cdots\wedge \tilde w_{N,0}\|_{\mVo^N}.
\ee
We track the time evolution of the above volume element as the perturbations $\tilde w_{i,0}$ evolve under the flow of $S'(t,w_0)$, according to the linear PDE
\be\la{leqcu}
\bal
\pa_t \tilde w_i(t)&=(-\bA+\mL(t))\tilde w_i(t),\quad t>0\\
\tilde w_i(t)_{|\pa\Omega}&=0,\quad t>0\\
\tilde w_i(0)&=\tilde w_{i,0}
\eal
\ee
where above, we denote (c.f. \eqref{LS1}, \eqref{LS2})
\be\la{leq}
\bal
\bA=&(-D_1\D,-D_2\D,\nu A)\\
\mL(t)=&\mL(t;w_0)=(\mL_1(t;w_0),\mL_2(t;w_0),\mL_u(t;w_0)).
\eal
\ee
We note that $\bA^{-1}$ (with homogeneous Dirichlet boundary conditions) is a compact, positive, self-adjoint operator on $\mVo$. It follows from \eqref{leqcu} that the squared volume element $V_N^2(t)=\|\tilde w_1(t)\wedge\cdots\wedge \tilde w_N(t)\|^2_{\mVo}$ evolves according to
\be\la{vv}
\bal
\fr{1}{2}\fr{d}{dt}V_N^2(t)=((-\bA_N+\mL_N)(\tilde w_1(t)\wedge\cdots\wedge\tilde w_N(t)),\tilde w_1(t)\wedge\cdots\wedge\tilde w_N(t))_{\mVo^N}
\eal
\ee
where $\mL_N:\mVo^N\to\mVo^N$ is the extension of $\mL$, defined component-wise (below, we omit the dependencies on time)
\be
\bal
\mL_N(\tilde w_1\wedge\cdots\wedge \tilde w_N)=&((\mL_1\tilde c_{1,1},\mL_2\tilde  c_{2,1},\mL_u\tilde u_1)\wedge\cdots\wedge(\tilde c_{1,N},\tilde c_{2,N},\tilde u_N))+\\
&\cdots\\
&+((\tilde c_{1,1},\tilde  c_{2,1},\tilde u_1)\wedge\cdots\wedge(\mL_1\tilde c_{1,N},\mL_2\tilde  c_{2,N},\mL_u\tilde u_N))
\eal
\ee
and similarly for $\bA_N$.

We now prove the following, which establishes exponential in time decay of volume elements:
\begin{prop}\la{expd}
For initial conditions $w_0\in \mV^+$, there exists $N^*$ depending on parameters and boundary conditions such that for every $N\ge N^*$ and every set of $N$ initial perturbations $\tilde w_{i,0}\in \mVo$, $i=1,...,N$, the corresponding volume element $V_N(t)$ decays exponentially:
\be
V_N(t)\le V_N(0)e^{-Nt},\quad t\ge \bar t
\ee
where $\bar t$ depends on parameters, boundary conditions, and $\|w_0\|_\mV$.
\end{prop}

\begin{proof} We invoke the following formula, proved in \cite{cf'}: for any linear operator $T:\mathcal{D}(T)\subset \mVo\to\mVo$, we have
\be\la{formula}
(T_N(\tilde w_1\wedge\cdots\wedge \tilde w_N),\tilde w_1\wedge\cdots \tilde w_N)_{\mVo^N}=\|\tilde w_1\wedge\cdots\wedge \tilde w_N\|_{\mVo^N}^2\tr(T P(\tilde w_1,...,\tilde w_N))
\ee
where $P(\tilde w_1,...,\tilde w_N)$ is defined to be the orthogonal projection from $\mVo$ onto the subspace spanned by $\tilde w_1,...,\tilde w_N\in \mVo$. It follows from the above formula and \eqref{vv} that $V_N^2(t)=\|\tilde w_1\wedge\cdots\wedge \tilde w_N\|_{\mVo}^2$ satisfies
\be
\fr{1}{2}\fr{d}{dt}V_N^2(t)=V_N^2(t)\tr((-\bA+\mL(t)) P(\tilde w_1(t),...,\tilde w_N(t)))
\ee
so that
\be\la{VN}
V_N^2(t)=V_N^2(0)\exp\left(2\int_0^t\tr((-\bA+\mL(s)) P(\tilde w_1(s),...,\tilde w_N(s)))\,ds\right).
\ee
Therefore, $V_N(t)$ is either identically zero for all $t\ge 0$ or strictly positive for all $t\ge 0$. Thus, without loss of generality, we assume that $V_N(0)>0$ so that $V_N(t)>0$ (and $\tilde w_1(t)\wedge\cdots\wedge \tilde w_N(t)\neq 0)$ for all time. 

To obtain the decay, it suffices to show that for all sufficiently large $N$, depending on parameters and boundary conditions, we have
\be 
\fr{1}{t}\int_0^t\tr(-\bA P(\tilde w_1(s),...,\tilde w_N(s)))\,ds+\fr{1}{t}\int_0^t\tr(\mL(s) P(\tilde w_1(s),...,\tilde w_N(s)))\,ds=:a_N(t)+l_N(t)\le -N
\ee
for all sufficiently large $t$, depending on parameters, boundary conditions, and $\|w_0\|_\mV$. We note that we require the above to hold for all choices of $N$ initial perturbations $\tilde w_{1,0},...,\tilde w_{N,0}\in\mVo$.

To estimate $a_N(t)$, we again use the formula \eqref{formula} to obtain
\be\la{an}
-a_N(t)=\fr{1}{t}\int_0^t\fr{(\bA_N(\tilde w_1(s)\wedge\cdots\wedge \tilde w_N(s)),\tilde w_1(s)\wedge\cdots\wedge\tilde w_N(s))_{\mVo^N}}{\|\tilde w_1(s)\wedge\cdots\wedge\tilde w_N(s)\|_{\mVo^N}^2}\,ds\ge \lambda_1+\cdots+\lambda_N
\ee
where we have denoted the eigenvalues of $\bA$ by $0<\lambda_1\le\lambda_2\le\cdots$ and used the fact that the smallest eigenvalue of $\bA_N$, then, is $\sum_{i=1}^N\lambda_i$.

Next, to estimate $l_N(t)$, for each $t>0$ we fix an orthonormal family $\rw_j(t)=(\rc_{1,j}(t),\rc_{2,j}(t),\ru_j(t))$, $j=1,...,N$ in $\mVo$ spanning the linear span of $\tilde w_1(t),...,\tilde w_N(t)$. We also denote $\rr_j=\rc_{1,j}-\rc_{2,j}$ and $\rp_j=\fr{1}{\ep}(-\D_D)^{-1}\rr_j.$ Then, since
\be
l_N(t)=\fr{1}{t}\int_0^t\sum_{j=1}^N(\mL(s)\rw_j(s),\rw_j(s))_{\mVo}\,ds
\ee
we focus on estimating the integrand in the integral above. We estimate the inner product, component by component, omitting below the dependence on time, 
\be
\bal
\left|\sum_{j=1}^N(\mL_i \rc_{i,j},\D\rc_{i,j})_{L^2}\right|
\lesssim&\sum_{j=1}^N\|\D \rc_{i,j}\|_{L^2}\times(\|\na c_i\|_{L^2}\|\na\rp_j\|_{L^\infty}+\|c_i\|_{L^6}\|\rr_j\|_{L^3}\\
&+\|\na \rc_{i,j}\|_{L^2}\|\na\Phi\|_{L^\infty}+\|\rc_{i,j}\|_{L^6}\|\rho\|_{L^3}+\|u\|_{L^\infty}\|\na\rc_{i,j}\|_{L^2}+\|\ru_j\|_{L^6}\|\na c_i\|_{L^3})\\
\le&\fr{D_i}{2}\sum_{j=1}^N\|\D\rc_{i,j}\|_{L^2}^2+C\sum_{j=1}^N\left(\sum_{k=1}^2\|\na\rc_{k,j}\|_{L^2}^2+\|\ru_j\|_V^2\right)b_i(t)\\
\le&\fr{D_i}{2}\sum_{j=1}^N\|\D\rc_{i,j}\|_{L^2}^2+CNb_i(t)
\eal
\ee
where
in the last line we used the fact that $\rc_{k,j}$ and $\ru_j$ are each orthonormal in $H_0^1$ and $V$, respectively, and we denote
\be
b_i(t)=\|\na c_i\|_{L^2}^2+\|c_i\|_{L^6}^2+\|\na\Phi\|_{L^\infty}^2+\|\rho\|_{L^3}^2+\|u\|_{L^\infty}^2+\|\na c_i\|_{L^3}^2.
\ee
We also point out that the final constant $C$ does not depend on $N$.

Similarly, we estimate
\be
\bal
\left|\sum_{j=1}^N(\mL_u\ru_j, A\ru_j)_H\right|\lesssim&\sum_{j=1}^N\|A\ru_j\|_H\left(\|\rho\|_{L^3}\|\na\rp_j\|_{L^6}+\|\rr_j\|_{L^3}\|\na\Phi\|_{L^6}\right)\\
\le& \fr{\nu}{2}\sum_{j=1}^N\|A\ru_j\|_H^2+C\sum_{j=1}^N\left(\sum_{k=1}^2\|\na \rc_{k,j}\|_{L^2}^2\right)b_u(t)\\
\le&\fr{\nu}{2}\sum_{j=1}^N\|A\ru_j\|_H^2+CNb_u(t)
\eal
\ee
where
\be
b_u(t)=\|\rho\|_{L^3}^2+\|\na\Phi\|_{L^6}^2.
\ee

Thus we have that
\be\la{EQ}
\bal
l_N(t)\le& \fr{1}{t}\int_0^t\sum_{j=1}^N\left(\fr{D_1}{2}\|\D\rc_{1,j}(s)\|_{L^2}^2+\fr{D_2}{2}\|\D\rc_{2,j}(s)\|_{L^2}^2+\fr{\nu}{2}\|A\ru_j(s)\|_H^2\right)\,ds\\
&+CN\fr{1}{t}\int_0^tb_1(s)+b_2(s)+b_u(s)\,ds.
\eal
\ee
At this point, let us note that due the formula \eqref{formula} again, we have
\be
\bal
\fr{(\bA_N(\tilde w_1\wedge\cdots\wedge \tilde w_N),\tilde w_1\wedge\cdots\wedge\tilde w_N)_{\mVo^N}}{\|\tilde w_1\wedge\cdots\wedge\tilde w_N\|_{\mVo^N}^2}=&\tr(\bA P(\tilde w_1,...,\tilde w_N))\\
=&\sum_{j=1}^N\left(\sum_{i=1}^2(D_i\D \rc_{i,j},\D\rc_{i,j})_{L^2}+(\nu A\ru_j,A\ru_j)_H\right)\\
=&\sum_{j=1}^N\left(\sum_{i=1}^2D_i\|\D \rc_{i,j}\|_{L^2}^2+\nu\|A\ru_j\|_H^2\right).
\eal
\ee
It follows from this equation, \eqref{EQ} and \eqref{an} that
\be\la{lN}
l_N(t)\le -\fr{1}{2}a_N(t)+CN\fr{1}{t}\int_0^t b_1(s)+b_2(s)+b_u(s)\,ds.
\ee
Thus, using \eqref{an} and \eqref{lN}, we have
\be\la{ll}
\bal
a_N(t)+l_N(t)\le& \fr{1}{2}a_N(t)+CN\fr{1}{t}\int_0^tb_1(s)+b_2(s)+b_u(s)\,ds\\
\le&-\fr{1}{2}(\lambda_1+\cdots+\lambda_N)+CN\fr{1}{t}\int_0^tb_1(s)+b_2(s)+b_u(s)\,ds.
\eal
\ee
Now we use the fact that
\be\la{evlb}
\lambda_j\ge cj^\fr{2}{3} \text{ for all } j\ge 1
\ee
for some constant $c>0$ (see Remark \ref{ev}) to conclude that
\be\la{aN}
a_N(t)+l_N(t)\le-k_1N^\fr{5}{3}+k_2N\fr{1}{t}\int_0^tb_1(s)+b_2(s)+b_u(s)\,ds
\ee
for some constants $k_1, k_2$ independent of $N$. Now, using Lemma \ref{l2!} together with Sobolev embeddings, we find that
\be\la{aN'}
a_N(t)+l_N(t)\le-k_1N^\fr{5}{3}+N(k_3+k_4\mF_0t^{-1})=-N(k_1N^\fr{2}{3}-k_3-k_4\mF_0t^{-1})
\ee
for constants $k_3, k_4$ depending only on parameters and boundary conditions. Thus, for all $t\ge \bar t=k_4\mF_0 k_3^{-1}$ and for all 
\be\la{N*}
N\ge N^*= \left\lceil((1+2k_3)k_1^{-1})^\fr{3}{2}\right\rceil
\ee
we have
\be
a_N(t)+l_N(t)\le -N.
\ee
This completes the proof.
\end{proof}

\begin{rem}\la{ev}
Lower bounds of the type \eqref{evlb} are well known for the Dirichlet eigenvalues of the operators $-\D$ and $A$ on three dimensional bounded domains (see for example \cite{cf}). Here we are using the fact that such lower bounds carry over to the product operator $\bA$, as is shown in \cite{abdo} by a counting argument.
\end{rem}

We now have all the ingredients to prove Theorem \ref{fdf}:

\begin{proof}
Using the same argument as in the proof of Theorem 14.15 in \cite{cf}, the conclusion of Theorem \ref{fdf} follows from the properties of the linear map $S'$ (Propositions \ref{ctsdiff} and \ref{S'}) and the exponential in time decay of sufficiently high dimensional volume elements under the flow of $S'$ (Proposition \ref{expd}). We refer the reader to \cite{cf} for all the details; here, we discuss only the main ideas of the proof. We fix small $r$, and we first cover $\mA$ with a finite number of balls $B(v_i,r_i)$, $i=1,...,k$ (here $B(v_i,r_i)$ denotes the ball in $\mV$, centered at $v_i$ with radius $r_i\le r$). Then, since $\mA\subset \mV_\gamma^+= \{(f,g,h)\in \mV^+\,|\, f_{|\pa\Omega}=\gamma_1, g_{|\pa\Omega}=\gamma_2\}$, we have
\be
\mA\subset \bigcup_{i=1}^k (B(v_i,r_i)\cap \mV_\gamma^+)
\ee
and from the invariance of $\mA$ it follows that
\be\la{sc!}
\mA\subset \bigcup_{i=1}^kS(t)(B(v_i,r_i)\cap \mV_\gamma^+)
\ee
for all $t\ge 0$. Next, using the fact that, for each $t$ and each $v_0\in\mA$, $S'(t,v_0)$ is a compact operator on $\mVo$, we define
\be
M(t,v_0)=\sqrt{S'(t,v_0)^*S'(t,v_0)}
\ee
where the superscript $*$ denotes the adjoint operator. Then, $M(t,v_0)$ is a compact, self-adjoint, nonnegative operator on $\mVo$. Furthermore, it is injective, a fact that follows from the injectivity of $S'(t,v_0)$. Thus $M(t,v_0)$ has a sequence of positive eigenvalues $m_j(t,v_0)$, counted with multiplicity, nonincreasing in $j$, and converging to $0$ as $j\to \infty$. The collection of orthonormal eigenvectors corresponding to the eignevalues $\{m_j(t,v_0)\}$ form an orthonormal basis for $\mVo$.

Now the main idea is to exploit the fact that the collection of ``balls" $\{B(v_i,r_i)\cap \mV_\gamma^+\}$, when transported by $S(t)$, \textit{for any} $t$, still comprises a cover for $\mA$, as per \eqref{sc!}. We fix $N^*>0$ large, from Proposition \ref{expd}, and we also fix $t$ sufficiently large (the largeness of $t$ required comes out of the proof). Then, assuming that $r$ is small enough, it follows from Proposition \ref{ctsdiff} that up to an error of order $rm_{N^*+1}(t,v_0)$, each ball $B(v_i,r_i)\cap \mV_\gamma^+$ is distorted by the map $S(t)$ into a set contained in an $N^*$-dimensional ellipsoid with semi axes of lengths $rm_j(t,v_0)$, $j=1,.,,,N^*$. We then make use of the fact that, due to the choice of $N^*$, both the volumes of these ellipsoids (which are on the order of $(2r)^{N^*}\Pi_{j=1}^{N^*}m_j(t,v_0))$ and the magnitude of the error $\approx rm_{N^*+1}(t,v_0)$ may be taken to be small for all $t$ large enough. This fact is a consequence of the exponential decay of volume elements, Proposition \ref{expd}.

 Then, having started with the $\mA$-cover $\{B(v_i,r_i)\}$, we may obtain a new cover of $\mA$, consisting of smaller balls, by the following procedure: to each $S(t)(B(v_i,r_i)\cap \mV_\gamma^+)$ corresponds an $N^*$-dimensional ellipsoid, which approximately covers it; we cover each of the $N^*$-dimensional ellipsoids by small balls of radius at most $\bar c r$ with $\bar c\in (0,1)$ and then uniformly dilate them so that this collection of dilated balls covers $\mA$. Because the error size $\approx rm_{N^*+1}(t,v_0)$ is small, a small dilation suffices, and we may ensure that the new collection of balls consists of balls strictly smaller than those in the original collection $\{B(v_i,r_i)\}$ - say, radius at most $cr$, with $c\in(\bar c,1)$ (in the precise proof, both $c$ and $\bar c$ are appropriate functions of $m_{N^*+1}(t,v_0)$; c.f. \cite{cf}).

An upper bound on the number of balls of radii at most $\bar c r$ needed to cover each $N^*$-dimensional ellipsoid is given as a function of the volume of these ellipsoids $\approx (2r)^{N^*}\Pi_{j=1}^{N^*}m_j(t,v_0)$ and the radius $\bar c r$. Then the product of the resulting upper bound with $n_\mA(r)$ (c.f. \eqref{nK}) gives an upper bound on $n_\mA(cr)$, a fact that follows from the procedure described in the preceding paragraph.

In the above considerations, we may take $r$ smaller and smaller, so that ultimately, we obtain information on the limiting behavior of the function $s\mapsto n_\mA(s)$ as $s\to 0^+$. Specifically, we conclude that for sufficiently large $D$, in terms of $N^*$ (c.f. \eqref{N*}), the function $s\mapsto s^Dn_\mA(s)$ satisfies $\limsup\limits_{s\to 0^+}s^Dn_\mA(s)=0$. In fact, the following choice of $D$ is sufficient:
\be
D= (2k_3+1)N^*
\ee
where $k_3$ is from \eqref{aN}. Thus we conclude (c.f. \eqref{fddef})
\be
d_f(\mA)\le (2k_3+1)N^*.
\ee
\end{proof}

\section{Space-time Averaged Electroneutrality}\la{EN}
In this section, we investigate the long time behavior of the local charge density $\rho$ in the singular limit of $\epsilon \to 0$. We establish electroneutrality $|\rho|\ll 1$ in a space-time averaged sense, for large times and small $\epsilon$.
\begin{thm}\la{aen}
For any global smooth solution of \eqref{np}-\eqref{stokes} with boundary conditions \eqref{gamma}-\eqref{noslip}, there exist a time $T^\ep$ (depending on parameters and boundary and initial data) and a constant $B_1$ (depending only on boundary data and parameters but not on initial data or $\ep$) such that for any $\tau\ge\epsilon^\fr{2}{3}$ and any $T\ge T^\ep$, we have
\be
\fr{1}{\tau}\int_T^{T+\tau}\left(\int_\Omega \rho_\ep^2(s,x)\,dx\right)\,ds\le B_1\ep^\fr{1}{3}.\la{BB1}
\ee
In particular, we have
\be
\limsup_{t\to\infty}\fr{1}{t}\int_0^t\left(\int_\Omega \rho_\ep^2(s,x)\,dx\right)\,ds\le B_1\ep^\fr{1}{3}.\la{BB2}
\ee\la{lten}
\end{thm}
Above, the subscript $\ep$ is denoted to emphasize the dependence of solutions on $\ep$. The proof of the theorem makes essential use of the following $L^\infty$ absorbing ball property proved in \cite{NPS}.
\begin{thm}\la{maxthm}
Suppose $(c_1,c_2 ,u)$ is a global smooth solution of \eqref{np}-\eqref{stokes} with boundary conditions \eqref{gamma}-\eqref{noslip}. Then, for all $\delta>0$, there exists $T_\delta$ depending on $\delta$, initial and boundary data, and parameters such that for all $t\ge T_\delta$ we have 
    \be\ug-\delta\le \um(t)\le\om(t)\le\og+\delta\ee
    where $\ug=\min_i\inf_{\pa\Omega}\gamma_i$ and $\og=\max_i\sup_{\pa\Omega}\gamma_i$, and
    \begin{align}\om(t)=\max_i\sup_{x\in\Omega} c_i(t,x),\quad \um(t)=\min_i\inf_{x\in\Omega} c_i(t,x).\end{align}
\end{thm} 

\noindent In particular, the theorem implies that for large times, $c_i$ obeys pointwise lower and upper bounds independent of $\ep$ and of initial data. One immediate consequence of this theorem, together with the Poisson equation \eqref{pois} and elliptic regularity, is the fact that, for all large times,
\be
\|\na \Phi(t)\|_{L^2}\le \fr{C}{\ep}.
\ee
Here, and for the remainder of the section, $C$ denotes a constant that does not depend on initial conditions \textit{nor} on $\ep$. This constant may differ from line to line.

Below, we show that by exploiting the dissipative structure of the NPS system, it is possible to improve on the $\ep$ dependence by a factor of $\ep^{-\fr{1}{2}}$:

\begin{prop}\la{phib}
For any global smooth solution of \eqref{np}-\eqref{stokes} with boundary conditions \eqref{gamma}-\eqref{noslip}, there exists a time $T^\ep$ depending on parameters and initial and boundary data such that for some $B_2$ independent of $\ep$ and of initial data, we have for all $t\ge T^\ep$
\be\la{phil2}
\|\na \Phi(t)\|_{L^2}\le\fr{B_2}{\ep^\fr{1}{2}}.
\ee
\end{prop}

\begin{proof}
We wish to analyze the long time behavior of the quantity $\|\na\Phi\|_{L^2}.$ To do so, we consider the time evolution of the energy
\be
\mathcal{E}=\sum_{i=1}^2\int_\Omega c_i^*\psi\left(\fr{c_i}{c_i^*}\right)\,dx+\fr{\epsilon}{2}\|\na(\Phi-\Phi^*)\|_{L^2}^2+\fr{1}{2K}\|u\|_H^2
\ee
where $\psi(s)=s\log s-s+1$ and $c_1^*,c_2^*>0$ and $\Phi^*$ are chosen to be a steady state solution of the Nernst-Planck equations
\begin{align}
\div(\na c_i^*+z_ic_i^*\na\Phi^*)=&0,\quad i=1,2\la{npstar}\\
-\ep\D\Phi^*=\rho^*=&c_1^*-c_2^*\la{poisstar}\\
{c_i^*}_{|\pa\Omega}=&\gamma_i,\quad i=1,2\\
{\Phi^*}_{|\pa\Omega}=&W.\la{phisb}
\end{align}
In \cite{NPS}, the existence of steady state solutions to the Nernst-Planck-Navier-Stokes equations with boundary conditions \eqref{gamma}-\eqref{noslip} is shown. A streamlined version of the proof then establishes the existence of steady state solutions of the above uncoupled Nernst-Planck system \eqref{npstar}-\eqref{phisb}. In general, it is not known if such steady states are unique, and for our purpose, we simply fix one steady state solution. All we make use of in the proceeding computations is the fact that $c_i^*,\Phi^*$ satisfy \eqref{npstar}-\eqref{phisb} and that they satisfy the following a priori bounds.
\begin{lemma}\la{ub*}
Suppose $c_1^*,c_2^*,\Phi^*$ satisfy \eqref{npstar}-\eqref{phisb}. Then 
\begin{enumerate}
\item $\inf_{\pa\Omega}\gamma_ie^{z_iW}=\lambda_i\le c_i^*e^{z_i\Phi^*}\le \Lambda_i=\sup_{\pa\Omega}\gamma_ie^{z_iW}$ on $\Omega$ for $i=1,2$
\item $\min\{\inf_{\pa\Omega}W,\log(\lambda_1/\Lambda_2)^\fr{1}{2}\}\le\Phi^*\le\max\{\sup_{\pa\Omega}W,\log(\Lambda_1/\lambda_2)^\fr{1}{2}\}$
\item $\ug\le c_i^*\le \og$ for $i=1,2$ 
\end{enumerate}
with $\ug,\og$ defined as in Theorem \ref{maxthm}.
\end{lemma}
The proof of this lemma may be found in \cite{NPS} (see Proposition 3 and Remark 6 of this reference). The main takeaway of the lemma is that $c_1^*,c_2^*,\Phi^*$ satisfy upper and lower bounds that are \textit{independent of} $\ep.$

Proceeding with the proof of Proposition \ref{phib}, we reformulate our equations in terms of the electrochemical potentials
\be
\mu_i=\log c_i+z_i\Phi,\quad \mu_i^*=\log c_i^*+z_i\Phi^*,\quad i=1,2.
\ee
Using these variables, the Nernst-Planck equations of the NPS system may be rewritten
\be
\pa_t c_i+u\cdot\na c_i=D_i\div(c_i\na \mu_i),\quad i=1,2\la{rr1}
\ee
and similarly \eqref{npstar} may be rewritten
\be
\div(c_i^*\na \mu_i^*)=0,\quad i=1,2.\la{rr2}
\ee
Using \eqref{rr1}, \eqref{rr2} we write the equation satisfied by $c_i-c_i^*$
\be
\pa_t(c_i-c_i^*)=-u\cdot\na c_i+D_i\div(c_i\na(\mu_i-\mu_i^*)+(c_i-c_i^*)\na\mu_i^*),\quad i=1,2.
\ee
Now we multiply the above by $\mu_i-\mu_i^*$ and integrate by parts. On the left hand side, we obtain, summing in $i$

\be
\bal
\sum_i\left(\pa_t(c_i-c_i^*),\mu_i-\mu_i^*\right)_{L^2}&=\sum_i\left(\fr{d}{dt}\int_\Omega c_i^*\psi\left(\fr{c_i}{c_i^*}\right)\,dx+z_i\left(\pa_t(c_i-c_i^*),\Phi-\Phi^*\right)_{L^2}\right)\\
&=\sum_i\fr{d}{dt}\int_\Omega c_i^*\psi\left(\fr{c_i}{c_i^*}\right)\,dx+\left(\pa_t(\rho-\rho^*),\Phi-\Phi^*\right)_{L^2}\\
&=\sum_i\fr{d}{dt}\int_\Omega c_i^*\psi\left(\fr{c_i}{c_i^*}\right)\,dx+\fr{\epsilon}{2}\fr{d}{dt}\|\na(\Phi-\Phi^*)\|_{L^2}^2.\la{lhs}
\eal
\ee
On the right hand side, for $i=1$, we have, 
\be
\begin{aligned}
&\left(-u\cdot\na c_1+D_1\div(c_1\na(\mu_1-\mu_1^*)+(c_1-c_1^*)\na\mu_1^*),\mu_1-\mu_1^*\right)_{L^2}\\
=&-\left(u\cdot\na  c_1,\mu_1-\mu_1^*\right)_{L^2}-D_1\int_\Omega c_1|\na(\mu_1-\mu_1^*)|^2\,dx-D_1\left((c_1-c_1^*)\na\mu_1^*,\na(\mu_1-\mu_1^*)\right)_{L^2}\la{int1}
\end{aligned}
\ee

From this point on, we assume enough time has passed so that $\fr{\ug}{2}\le c_i\le \fr{3\og}{2},$ $i=1,2.$ That is, we restrict ourselves to $t\ge T_\fr{\ug}{2}$ where $T_\fr{\ug}{2}$ is obtained from Theorem \ref{maxthm} by taking $\delta=\ug/2.$ Then from \eqref{int1}, using a Young's inequality, we obtain

\be
\begin{aligned}
&\left(-u\cdot\na c_1+D_1\div(c_1\na(\mu_1-\mu_1^*)+(c_1-c_1^*)\na\mu_1^*),\mu_1-\mu_1^*\right)_{L^2}\\
\le &-\left(u\cdot\na  c_1,\mu_1-\mu_1^*\right)_{L^2}-\fr{D_1}{2}\int_\Omega c_1|\na(\mu_1-\mu_1^*)|^2\,dx+C\int_\Omega c_1^*|\na \mu_1^*|^2\,dx.\la{int2}
\end{aligned}
\ee
The constant $C>0$ above includes lower and upper bounds on $c_i, c_i^*$ in terms of $\ug,\og$. We now take a closer look at the term involving $u$. We have, integrating by parts, using Hölder and Young's inequalities, bounds on $c_i, c_i^*$, the Poincaré bound $\|u\|_H\le C\|u\|_V$ and the fact that $\div u=0$,
\be
\bal
-\left(u\cdot\na c_1,\mu_1-\mu_1^*\right)_{L^2}=&-\left(u\cdot\na c_1,\log c_1+\Phi\right)_{L^2}+\left(u\cdot\na c_1,\mu_1^*\right)_{L^2}\\
=&-\left(u,\na(c_1\log c_1-c_1)\right)_{L^2}+\int_\Omega uc_1\cdot\na\Phi\,dx-\left(u c_1,\na\mu_1^*\right)_{L^2}\\
\le& \int_\Omega uc_1\cdot\na\Phi\,dx+\fr{\nu}{4K}\|u\|_V^2+C\int_\Omega c_1^*|\na\mu_1^*|^2\,dx.\la{int3}
\eal
\ee
The constant $C$ depends on bounds on $c_i,c_i^*$ only through $\ug,\og$. We now estimate the integral $\int_\Omega c_1^*|\na\mu_1^*|^2\,dx$ that occurs both in \eqref{int2} and \eqref{int3}. To this end, we take $\tilde\mu_1$ to be the unique solution to 
\be
\D \tilde\mu_1=0 \text{ in } \Omega,\quad \tilde\mu_1=\log\gamma_1+W \text{ on } \pa\Omega.
\ee
Then, multiplying \eqref{rr2} by $\mu_1^*-\tilde\mu_1$ and integrating by parts, we obtain
\be
\int_\Omega c_1^*|\na \mu_1^*|^2\,dx=\int_\Omega c_1^*\na\mu_1^*\cdot\na\tilde\mu_1\,dx
\ee
from which we obtain, after a Cauchy–Schwarz and Young's inequality,
\be
\int_\Omega c_1^*|\na\mu_1^*|^2\,dx\le C\la{int4}
\ee
for some constant $C$ independent of $\epsilon.$ Thus, returning to \eqref{int2} and using \eqref{int3} and \eqref{int4}, we obtain
\be
\bal
&\left(-u\cdot\na c_1+D_1\div(c_1\na(\mu_1-\mu_1^*)+(c_1-c_1^*)\na\mu_1^*),\mu_1-\mu_1^*\right)_{L^2}\\
\le& C-\fr{D_1}{2}\int_\Omega c_1|\na(\mu_1-\mu_1^*)|^2\,dx+\int_\Omega uc_1\cdot\na\Phi\,dx+\fr{\nu}{4K}\|u\|_V^2.\la{int5}
\eal
\ee
Analogous computations for $i=2$ give us
\be
\bal
&\left(-u\cdot\na c_2+D_2\div(c_2\na(\mu_2-\mu_2^*)+(c_2-c_2^*)\na\mu_2^*),\mu_2-\mu_2^*\right)_{L^2}\\
\le& C-\fr{D_2}{2}\int_\Omega c_2|\na(\mu_2-\mu_2^*)|^2\,dx-\int_\Omega uc_2\cdot\na\Phi\,dx+\fr{\nu}{4K}\|u\|_V^2.\la{int6}
\eal
\ee
Then combining \eqref{int5} and \eqref{int6} with \eqref{lhs}, we obtain
\be
\bal
&\fr{d}{dt}\left(\sum_{i=1}^2\int_\Omega c_i^*\psi\left(\fr{c_i}{c_i^*}\right)\,dx+\fr{\ep}{2}\|\na(\Phi-\Phi^*)\|_{L^2}^2\right)+\sum_{i=1}^2\fr{D_i}{2}\int_\Omega c_i|\na(\mu_i-\mu_i^*)|^2\,dx\\
\le& C+\int_\Omega u\rho\cdot\na\Phi\,dx+\fr{\nu}{2K}\|u\|_V^2.\la{int7}
\eal
\ee
To close the estimates, we take the inner product of \eqref{stokes} with $\fr{1}{K}u$ and integrate by parts,
\be
\fr{1}{2K}\fr{d}{dt}\|u\|_H^2+\fr{\nu}{K}\|u\|_V^2=-\int_\Omega u\rho\cdot\na\Phi\,dx.\la{int8}
\ee
Thus, adding \eqref{int8} to \eqref{int7}, we obtain
\be\la{INT}
\fr{d}{dt}\mathcal{E}+\sum_{i=1}^2\fr{D_i}{2}\int_\Omega c_i|\na(\mu_i-\mu_i^*)|^2\,dx+\fr{\nu}{2K}\|u\|_V^2\le C.
\ee
Poincaré's inequality gives us the following control, $\|u\|_H\le C\|u\|_V.$ It remains to obtain an appropriate lower bound for the dissipation terms $\int_\Omega c_i|\na(\mu_i-\mu_i^*)|^2\,dx.$ Due to the pointwise lower and upper bounds on $c_i, c_i^*$, the following log-Sobolev type inequality is available,
\be\la{LS}
\sum_{i=1}^2\int_\Omega c_i^*\psi\left(\fr{c_i}{c_i^*}\right)\,dx+\fr{\ep}{2}\|\na(\Phi-\Phi^*)\|_{L^2}^2\le C\sum_{i=1}^2\fr{D_i}{2}\int_\Omega c_i|\na(\mu_i-\mu_i^*)|^2\,dx
\ee
where $C$ depends only $D_i$, the diameter of the domain, and lower and upper bounds of $c_i,c_i^*$ (given in terms of $\ug$ and $\og$ since $t\ge T_\fr{\ug}{2})$. We refer the reader to \cite{NPS} for a proof of inequality \eqref{LS}. Using this inequality in \eqref{INT}, we obtain for $t\ge T_\fr{\ug}{2}$,
\be
\fr{d}{dt}\mathcal{E}+C\mathcal{E}\le C.\la{Eineq}
\ee
Thus, by applying a Grönwall inequality to \eqref{Eineq}, we find
\be
\mathcal{E}(t)\le \mathcal{E}(T_\fr{\ug}{2})e^{-C(t-T_\fr{\ug}{2})}+C, \quad t\ge T_\fr{\ug}{2}.\la{ET}
\ee
Next, since $\mathcal{E}(T_\fr{\ug}{2})e^{-C(t-T_\fr{\ug}{2})}\searrow 0$ as $t\to\infty$, we choose $T^\ep\ge T_\fr{\ug}{2}$, depending on $\mathcal{E}(T_\fr{\ug}{2})$, so that for all $t\ge T^\ep$, we have $\mathcal{E}(T_\fr{\ug}{2})e^{-C(t-T_\fr{\ug}{2})}\le C$ for some $C$ independent of $\ep$ and initial data. We observe that the size of $\mathcal{E}(T_\fr{\ug}{2})$ can be estimated by initial data using \eqref{FB} (with $t=0$) and the elementary inequality $\psi(x)\le\max\{1,(x-1)^2/2\}$ applied to $x=c_i/c_i^*$.

Thus, we have shown that
\be
\|\na(\Phi(t)-\Phi^*)\|_{L^2}^2\le\fr{2}{\ep}\mathcal{E}(t)\le \fr{C}{\ep},\quad t\ge T^\ep\la{phibb}
\ee
for $C$ independent of $\ep$ and initial data.

Lastly, taking $\tilde \Phi$ such that 
\be
\D\tilde \Phi=0 \text{ in } \Omega, \quad \tilde \Phi=W \text{ on } \pa\Omega
\ee 
we multiply \eqref{poisstar} by $\Phi^*-\tilde \Phi$ and integrate by parts to find,
\be
\ep\int_\Omega |\na\Phi^*|^2\,dx=\ep\int_\Omega \na\Phi^*\cdot\na\tilde\Phi\,dx+\int_\Omega\rho^*(\Phi^*-\tilde\Phi)\,dx.\la{eee}
\ee
From Lemma \ref{ub*}, we have that $\rho^*$ and $\Phi^*$ are both bounded by constants independent of $\epsilon$. Thus we have
\be
\fr{\ep}{2}\int_\Omega|\na\Phi^*|^2\,dx\le \fr{\ep}{2}\int_\Omega|\na\tilde\Phi|^2\,dx+C
\ee
and we conclude that
\be
\|\na\Phi^*\|_{L^2}\le \fr{C}{\ep^\fr{1}{2}}
\ee
for some $C$ independent of $\ep$. Combining this last estimate with \eqref{phibb} gives us the desired result \eqref{phil2}.
\end{proof}

Contained in the preceding proof of Proposition \ref{phib} is the proof of the following proposition
\begin{prop}
For all $t\ge T^\ep$ (with $T^\ep$ determined by Proposition \ref{phib}), we have
\be
\|u(t)\|_H\le B_3
\ee
for $B_3>0$ independent of $\ep$ and initial conditions.\la{B4}
\end{prop}
\begin{proof}
Recalling that $\fr{1}{2K}\|u(t)\|_H^2\le \mathcal E(t)$, the statement of the proposition follows from \eqref{ET} and the choice of $T^\ep$.
\end{proof}

We now prove the main result of this section, Theorem \ref{lten}:

\begin{proof}
We define $\Gamma_i$, the harmonic extension of $\gamma_i$, satisfying
\be
\D\Gamma_i=0 \text{ in } \Omega,\quad \Gamma_i=\gamma_i \text{ on } \pa\Omega.\la{Gamma}
\ee
We restrict ourselves to $t\ge T^\ep$, where $T^\ep$ is determined by Proposition \ref{phib}. We recall that in particular $t\ge T^\ep$ implies $\fr{\ug}{2}\le c_i(t)\le \fr{3\og}{2}$ (c.f. paragraph below \eqref{int1}) and $\|\na\Phi(t)\|_{L^2}\le \fr{B_2}{\ep^\fr{1}{2}}$ (Proposition \ref{phib}).

We multiply \eqref{np} by $\fr{1}{D_i}\log\fr{c_i}{\Gamma_i}$ and integrate by parts to obtain
\be
\fr{1}{D_i}\fr{d}{dt}\int_\Omega \Gamma_i\psi\left(\fr{c_i}{\Gamma_i}\right)\,dx= I_1^i+I_2^i+I_3^i\la{III}
\ee
where
\begin{align}
    I_1^i&=\int_\Omega \D c_i\log\fr{c_i}{\Gamma_i}\,dx\\
    I_2^i&=-D_i^{-1}\int_\Omega u\cdot\na c_i\log\fr{c_i}{\Gamma_i}\,dx\\
    I_3^i&=z_i\int_\Omega \div(c_i\na\Phi)\log\fr{c_i}{\Gamma_i}\,dx.
\end{align}
We estimate $I_1^i$ by integrating by parts and using the bound $c_i\le\fr{3\og}{2}$ and a Young's inequality,
\be
I_1^i=-\int_\Omega \fr{|\na c_i|^2}{c_i}\,dx+\int_\Omega \fr{\na c_i}{c_i^\fr{1}{2}}\cdot(\na\log\Gamma_i)c_i^\fr{1}{2}\,dx\le -\fr{1}{2}\int_\Omega \fr{|\na c_i|^2}{c_i}\,dx+C.
\ee
To estimate $I_2^i$, we first note that 
\be
\int_\Omega u\cdot\na c_i\log c_i\,dx=\int_\Omega u\cdot\na(c_i\log c_i-c_i)\,dx=0
\ee
so that, using $c_i\le \fr{3\og}{2}$ and Proposition \ref{B4},
\be
I_2^i\le D_i^{-1}\left|\int_\Omega u\cdot\na c_i\log\Gamma_i\,dx\right|=D_i^{-1}\left|\int_\Omega uc_i\cdot\na\log\Gamma_i\,dx\right|\le C\|u\|_H\le C.
\ee
Lastly we estimate $I_3^i$. Integrating by parts, we have
\be
\bal
I_3^i&=-z_i\int_\Omega\na c_i\cdot\na\Phi\,dx+z_i\int_\Omega c_i\na\Phi\cdot\na\log\Gamma_i\,dx\\
&= J_1^i+J_2^i.
\eal
\ee
First, using $c_i\le\fr{3\og}{2}$, we have, by Proposition \ref{phib}
\be
J_2^i\le C\|\na\Phi\|_{L^2}\le \fr{C}{\ep^\fr{1}{2}}.\la{J2}
\ee
To estimate $J_1^i$, we construct another extension of the boundary data $\gamma_i$ to $\Omega$ in the following way. We fix a family of smooth, nonnegative cutoff functions $\chi^\ep:\bar\Omega\to\mathbb{R}$ so that $0\le\chi^\ep\le 1,$ ${\chi^\ep}_{|\pa\Omega}=1$ and $\chi^\ep(x)=0$ for all $x\in \Omega$ such that $d(x,\pa\Omega)=\inf_{y\in\pa\Omega}|x-y|\ge \ep^\fr{1}{3}.$ We further require that $|\na\chi^\ep|\le \fr{C}{\ep^\fr{1}{3}}$, where the constant $C$ depends on the domain but not on $\ep.$ Then, the function $\Gamma_i^\ep$ defined by $\Gamma_i^\ep=\chi^\ep\Gamma_i$ is an extension of $\gamma_i$ that vanishes away from the boundary $\pa\Omega$ at distances larger than $\ep^\fr{1}{3}.$ Furthermore, we have $|\Gamma_i^\ep|\le \og$ and $|\na\Gamma_i^\ep|\le |\Gamma_i\na\chi^\ep|+|\chi^\ep\na\Gamma_i|\le \fr{C}{\ep^\fr{1}{3}}.$ It follows from the latter bound and the fact that $\Gamma_i^\ep$ is supported on a set of measure of order $\ep^\fr{1}{3}$ that 
\be
\|\na\Gamma_i^\ep\|_{L^2}=\left(\int_\Omega |\na\Gamma_i^\ep|^2\,dx\right)^\fr{1}{2}\le C(\ep^\fr{1}{3} \ep^{-\fr{2}{3}})^\fr{1}{2}=C\ep^{-\fr{1}{6}}.\la{16}
\ee

Now we estimate
\be
\bal
J_1^i&=-z_i\int_\Omega \na c_i\cdot\na\Phi\,dx\\
&=-z_i\int_\Omega \na(c_i-\Gamma_i^\ep)\cdot\na\Phi\,dx-z_i\int_\Omega \na \Gamma_i^\ep\cdot\na\Phi\,dx\\
&\le -\fr{z_i}{\epsilon}\int_\Omega c_i\rho\,dx+\fr{z_i}{\epsilon}\int_\Omega \Gamma_i^\ep\rho\,dx+\|\na\Gamma_i^\ep\|_{L^2}\|\na\Phi\|_{L^2}\\
&\le -\fr{z_i}{\ep}\int_\Omega c_i\rho\,dx+\fr{1}{2\ep}\int_\Omega \rho^2\,dx+\fr{1}{2\ep}\int_\Omega(\Gamma_i^\ep)^2 \,dx+\fr{C}{\ep^\fr{2}{3}}\\
&\le-\fr{z_i}{\ep}\int_\Omega c_i\rho\,dx+\fr{1}{2\ep}\int_\Omega \rho^2\,dx+\fr{C}{\ep^\fr{2}{3}}
\eal
\ee
where in the fourth line, we used Proposition \ref{phib} and \eqref{16}, and in the last line, we used the fact that $\Gamma_i^\ep$ is supported on a set of measure of order $\ep^\fr{1}{3}.$ Now, we collect all our estimates for $I_1^i, I_2^i, I_3^i$ and sum in $i$ to obtain from \eqref{III},
\be
\fr{d}{dt}\sum_{i=1}^2\fr{1}{D_i}\int_\Omega \Gamma_i\psi\left(\fr{c_i}{\Gamma_i}\right)\,dx+\fr{1}{2}\sum_{i=1}^2\int_\Omega \fr{|\na c_i|^2}{c_i}\,dx+\fr{1}{2\ep}\int_\Omega\rho^2\,dx\le\fr{C}{\ep^\fr{1}{2}}+\fr{C}{\ep^\fr{2}{3}}+C\le \fr{C}{\ep^\fr{2}{3}}.\la{gcgc}
\ee
Then, for any $\tau\ge \epsilon^\fr{2}{3}$ and $T\ge T^\ep$, integrating \eqref{gcgc} from $T$ to $T+\tau$, we have
\be
\fr{1}{2\ep}\int_T^{T+\tau}\int_\Omega\rho^2\,dx\,ds\le \sum_{i=1}^2\fr{1}{D_i}\int_\Omega \Gamma_i\psi\left(\fr{c_i(T)}{\Gamma_i}\right)\,dx+\fr{C\tau}{\ep^\fr{2}{3}}.\la{!}
\ee
Since for times larger than or equal to $T^\ep$, $c_i$ obeys $\fr{\ug}{2}\le c_i\le \fr{3\og}{2}$, we obtain from \eqref{!}
\be\la{204}
\fr{1}{\tau}\int_T^{T+\tau}\int_\Omega\rho^2\,dx\,ds\le \fr{C\ep}{\tau}+C\ep^\fr{1}{3}\le C\ep^\fr{1}{3},\quad \tau\ge \ep^\fr{2}{3}.
\ee
Thus we have shown \eqref{BB1}. To show \eqref{BB2}, we compute
\be\bal
\limsup_{t\to\infty}\fr{1}{t}\int_0^t\|\rho(s)\|_{L^2}^2\,ds\le& \limsup_{t\to\infty}\fr{1}{t}\int_0^{T^\epsilon}\|\rho(s)\|_{L^2}^2 +\limsup_{t\to\infty}\fr{1}{t}\int_{T^\epsilon}^{t}\|\rho(s)\|_{L^2}^2\,ds\\
\le&0+\limsup_{t\to\infty}\fr{1}{t-T^\epsilon}\int_{T^\epsilon}^t\|\rho(s)\|_{L^2}^2\,ds\\
\le&C\epsilon^\fr{1}{3}
\eal\ee
where in the last line we used \eqref{204}. This completes the proof.
\end{proof}

\end{document}